\documentclass{amsart}

\usepackage[utf8]{inputenc}
\usepackage{tikz}
\usetikzlibrary{matrix}
\usepackage{amsfonts, amssymb, amsmath, amsthm}
\usepackage{graphicx}
\usepackage{float, comment}
\usepackage[hidelinks]{hyperref}
\usepackage{verbatim}
\usepackage{mathtools, nccmath}

\usepackage{enumitem}
\usepackage{verbatim}

\usepackage{centernot}

\newtheorem{teor}{Theorem}[section]
\newtheorem{lema}[teor]{Lemma}
\newtheorem{prop}[teor]{Proposition}
\newtheorem{fact}[teor]{Fact}

\newtheorem{defin}[teor]{Definition}
\newtheorem*{claim}{Claim}
\newtheorem*{prop*}{Proposition}

\newtheorem{cor}[teor]{Corollary}
\newtheorem{remark}[teor]{Remark}

\newtheorem{question}[teor]{Question}
\newtheorem{external claim}[teor]{Claim}

\newcommand{\bslant}[2]{{\raisebox{.2em}{$#1$}/\raisebox{-.2em}{$#2$}}}

\newcommand{\C}{{\mathfrak C}}
\newcommand{\R}{\mathbb{R}}
\newcommand{\mL}{\mathcal{L}}
\newcommand{\I}{\mathcal{I}}
\newcommand{\II}{\mathbf{I}}
\newcommand{\fwapp}{F'_{\textrm{WAP}}}
\newcommand{\fwap}{F_{\textrm{WAP}}}
\newcommand{\ftaa}{F'_{\textrm{Tame}}}
\newcommand{\fta}{F_{\textrm{Tame}}}

\DeclareMathOperator{\tp}{{tp}}

\DeclareMathOperator{\aut}{{Aut}}

\DeclareMathOperator{\Sym}{{Sym}}
\DeclareMathOperator{\End}{{End}}
\DeclareMathOperator{\Image}{{Im}}
\DeclareMathOperator{\cl}{{cl}}

\newcommand{\orcidlogo}{\includegraphics[height=\fontcharht\font`\B]{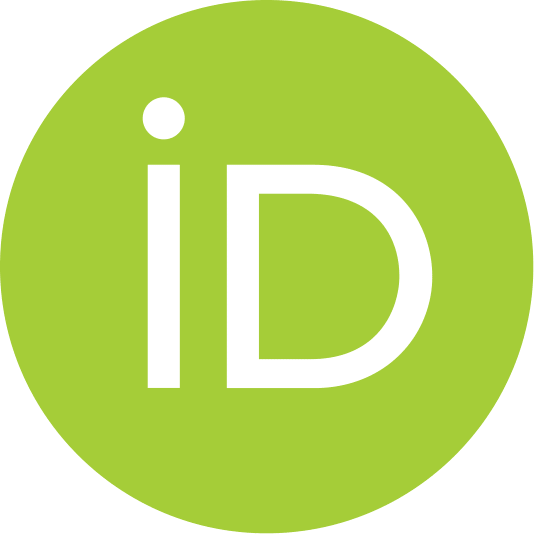}}
\newcommand{\orcid}[1]{\href{#1}{\orcidlogo #1}}

\usepackage{stackengine}
\stackMath
\usepackage[
backend=bibtex,
style=alphabetic,
]{biblatex}
\title{Maximal WAP and tame quotients of type spaces}

\author{Krzysztof Krupi\'{n}ski}
\author{Adri\'{a}n Portillo}

\thanks{\noindent Both authors are supported by the Narodowe Centrum Nauki grant no. 2016/22/E/ST1/00450.} 
\address{Instytut Matematyczny Uniwersytetu Wroc{\l}awskiego, pl. Grunwaldzki 2, 50-384 Wroc{\l}aw, Poland}
\address{Krzysztof Krupi\'{n}ski \orcid{https://orcid.org/0000-0002-2243-4411}}
\email{Krzysztof.Krupinski@math.uni.wroc.pl}
\address{Adri\'{a}n Portillo \orcid{https://orcid.org/0000-0001-9354-8574}}
\email{Adrian.Portillo-Fernandez@math.uni.wroc.pl}

\address{Instytut Matematyczny Uniwersytetu Wroc{\l}awskiego, pl. Grunwaldzki 2, 50-384 Wroc{\l}aw, Poland}


\begin{document}
\begin{abstract}
    We study maximal WAP and tame (in the sense of topological dynamics) quotients of $S_X(\C)$, where $\C$ is a sufficiently saturated (called monster) model of a complete theory $T$, $X$ is a $\emptyset$-type-definable set, and $S_X(\C)$ is the space of complete types over $\C$ concentrated on $X$. Namely, let $\fwap\subseteq S_X(\C)\times S_X(\C)$ be the finest  closed, $\aut(\C)$-invariant equivalence relation on $S_X(\C)$ such that the flow $( \aut(\C), S_X(\C)/\fwap )$ is WAP, and let $\fta\subseteq S_X(\C)\times S_X(\C)$ be the finest closed, $\aut(\C)$-invariant equivalence relation on $S_X(\C)$ such that the flow $( \aut(\C), S_X(\C)/\fta )$ is tame. We show good behaviour of $\fwap$ and $\fta$ under changing the monster model $\C$. Namely, we prove that if $\C'\succ \C$ is a bigger monster model, $\fwapp$ and $\ftaa$ are the counterparts of $\fwap$ and $\fta$ computed for $\C'$, and $r\colon S_X(\C')\to S_X(\C)$ is the restriction map, then $r[\fwapp]=\fwap$ and $r[\ftaa]=\fta$. Using these results, we show that the Ellis (or ideal) groups of $( \aut(\C), S_X(\C)/\fwap )$ and  $( \aut(\C), S_X(\C)/\fta )$ do not depend on the choice of the monster model $\C$.
\end{abstract}

\maketitle

\section{Introduction}

The project is to study maximal WAP and tame quotients (in the sense of topological dynamics) of type spaces over sufficiently saturated models (in the sense of model theory). In this paper, we show the fundamental property of ``model-theoretic absoluteness'' of Ellis groups of such quotients, i.e., that they do not depend on the model for which they are computed.

Recall that a {\em flow} is pair $(G,X)$, where $G$ is a topological group acting continuously on a compact (Hausdorff) space $X$. 
%In many considerations, the topology on $G$ does not matter, so can be taken to be discrete and then $(G,X)$ being a flow means that the action is by homeomorphisms. 
The essential for this paper terms appearing in this introduction are defined in Section \ref{section: preliminaries}.  

{\em Weakly almost periodic} (or {\em WAP}) {\em flows} form a well-behaved family, playing a fundamental role in topological dynamics with applications to other areas, e.g. to ergodic theory (see \cite{55bad191-0781-3b5e-a214-06bd2e8fc797}). The minimal WAP flows are known to be equicontinuous and have been been classified as homogeneous spaces for compact groups \cite[Chapter 3, Theorem 6]{auslander1988minimal}.
% (on which G acts as subgroups of the compact groups in question).

{\em Tame flows}, originating in \cite{Koh95}, form a wider important family, intensively studied in recent two decades by Glasner, Megrelishvili, and others (see e.g. \cite{glasner2007structure, glasner2013representations, glasner2018structure, Glasner2018}). In the case of metrizibale flows, tameness corresponds to one of the two situations in the so-called dynamical Bourgain, Fremlin, Talagrand dichotomy (see \cite[Theorem 3.2]{glasner2004hereditarily}). A deep structure theorem on tame, metrizable, minimal flows was obtained by Glasner in \cite{glasner2018structure}. Tame flows are related to several other areas, e.g. Banach spaces, circularly ordered systems, substitutions and tilings, quasicrystals, cut and project schemes (so to the so-called theory of aperiodic order); see \cite{aujogue2015ellis} and the survey \cite{glasner2013representations}.

From the point of view of this paper, it is important that there exist direct correspondences between model-theoretic {\em stability} and WAP flows (discovered in \cite{yaacov2014model, ben2016weakly}, see also \cite{conant2021stability, hrushovski2021amenability}), and between model-theoretic {\em NIP} and tame flows (discovered independently in \cite{definably_Amenable_nip}, \cite{ibarlucia2016dynamical}, and \cite{Khanaki2020-KHASTN}, and further developed e.g. in \cite{Krupinski2018GaloisGA} and \cite{codenotti2023ranks}). One of the well-known forms of these kind of correspondences is that stability [resp. NIP] of a theory $T$ is equivalent to saying that all flows of the form $(\aut(M),S(M))$ (where $M$ is a model of $T$, $\aut(M)$ is the group of automorphisms of $M$, and $S(M)$ is the space of complete types over $M$ in a fixed tuple of variables) are WAP [resp. tame] (e.g., see \cite[Section 5]{Krupinski2018GaloisGA}).

Stability theory is the core of model theory \cite{shelah1990classification, pillay1996geometric}, and in the last three decades model theory has been focusing on extending the context of stability theory to wider classes of theories, which often requires new ideas and tools, and has applications to other branches of mathematics (e.g. to algebraic geometry and additive combinatorics). And one of the most deeply studied classes of theories extending stable theories is the class of theories with NIP (the non independence property); e.g., see \cite{Shelah2005StronglyDT, MR2373360, hrushovski2011nip, definably_Amenable_nip}.

Around 2005, Newelski came up with an idea of using tools from topological dynamics in order to extend some aspects of stability theory to much wider unstable situations \cite{10.2178/jsl/1231082302}. The point is that various spaces of types can be naturally treated as flows and so become objects of topological dynamics. Since then, the topic has been broadened by a multitude of authors: Ben Yaacov, Chernikov, Hrushovski, Krupi\'nski, Newelski, Pillay, Rzepecki, Simon, and others (e.g., \cite{definably_Amenable_nip, ben2016weakly, Krupiski2017BoundednessAA, KPR18, 959fd4d1-73ad-3b7b-bc69-e998be32d2cf}). This led not only to essentially new results in model theory (e.g. on the complexity of strong types in \cite{KPR18,Krupinski2018GaloisGA}) but also to applications to additive combinatorics in \cite{krupiński2023generalized}.

Both in topological dynamics and in model theory, it is natural to study quotients by suitable equivalence relations. In topological dynamics, these are quotients of flows by closed, invariant equivalence relations (equivalently, homomorphic images of flows). In model theory, these are quotients of a sufficiently saturated model (or a subset of such a model) by nice (definable, type-definable, invariant, ...) equivalence relations.  This includes strong types, hyperimaginary sorts, or quotients of definable groups by various model-theoretic connected components. A particular role is played by quotients by bounded (i.e. with few classes comparing to the degree of saturation of the model) equivalence relations; see e.g. \cite{lascar2001hyperimaginaries, MR2373360, hrushovski2011nip, KPR18}. In recent papers \cite{MR3796277, 10.1215/00294527-2022-0023, krupinski2023maximal}, also maximal stable quotients entered the picture, mostly in the NIP context.

In this paper, we study maximal WAP and tame quotients of flows of the form $(\aut(\C),S_X(\C))$, where  $\C$ is a sufficiently saturated model, $X$ is a $\emptyset$-type-definable set, and $S_X(\C)$ is the space of complete types over $\C$ concentrated on $X$. 
%Topological dynamics methods were introduced in model theory by Newelski \cite{10.2178/jsl/1231082302, 10.1112/jlms/jdr075}, with the goal to extend results from stable group theory to the unstable context. Since then, the topic has been broadened by a multitude of authors: Chernikov, Hrushovski, Krupi\'nski, Newelski, Pillay, Rzepecki, Simon, and others (e.g. \cite{definably_Amenable_nip,  Krupiski2017BoundednessAA, KPR18, 959fd4d1-73ad-3b7b-bc69-e998be32d2cf}). 

As mentioned above, in various contexts, stability corresponds to WAP and NIP to tameness. We check in Section \ref{section: Stable vs WAP} that indeed the quotient of a $\emptyset$-type-definable set $X$ by a $\emptyset$-type-definable equivalence relation is stable [resp. NIP] if and only if the quotient of the type space $S_X(\C)$ by the corresponding closed, $\aut(\C)$-invariant equivalence relation is WAP [resp. tame]. However, in Proposition \ref{fwap strictly finer}, we observe that for an arbitrary $\emptyset$-type-definable non-stable [resp. IP] set $X$, the finest closed, $\aut(\C)$-invariant equivalence relation on $S_X(\C)$ with WAP [resp. tame] quotient is never induced by a $\emptyset$-type-definable equivalence relation on $X$ (necessarily with stable [resp. NIP] quotient). Thus, one can expect that the maximal WAP [resp. tame] quotient of $X$ captures more information about the theory in question than the quotient induced by the finest $\emptyset$-type-definable, stable [resp. NIP] equivalence relation on $X$.

The {\em Ellis} (or {\em ideal}) {\em groups} of flows play a very important role both in abstract topological dynamics (for example, in general structural theorems about minimal flows (e.g., see \cite{auslander1988minimal, glasner2018structure})) and in model theory (for example, to get new information about model-theoretic invariants such as $G/G^{000}$ or the Lascar Galois group of the theory $T$ denoted by $\textrm{Gal}_L(T)$ (see \cite{KRUPIŃSKI_PILLAY_2017, KPR18, Krupinski2018GaloisGA})), and in recent applications to additive combinatorics in \cite{krupiński2023generalized}. From the model-theoretic perspective,  the so-called Newelski's conjecture (or Ellis group conjecture) turned out to be particularly influential (e.g., see \cite{10.2178/jsl/1231082302, gismatullin2012some, definably_Amenable_nip, KRUPIŃSKI_PILLAY_2017}).

A good model-theoretic notion should be independent of the choice of the (sufficiently saturated) model in which it is defined. If this happens to be true, we say that the notion  is (model-theoretically) {\em absolute}; one can also say that it is an {\em invariant} of the theory in question. 

The main result of \cite{ Krupiski2017BoundednessAA} says that the Ellis group of $S_X(\C)$ is absolute (i.e., does not depend on $\C$). The main theorem of this paper says that the same is true for the Ellis groups of the maximal WAP and tame quotients of $S_X(\C)$. Below are detailed statements; the relevant definitions are given in further sections.

Let $\C \prec \C'$ be models of a complete theory $T$ which are $\kappa$-saturated and strongly $\kappa$-homogeneous, where $\kappa$ is specified in the statements below. % such that $\C$ is at least $\kappa$-saturated 
%and strongly $\kappa$-homogeneous 
%with a strong limit cardinal $\kappa>\lvert T\rvert$,  and $\C'$ is $\kappa'$-saturated and strongly $\kappa'$-homogeneous with a strong limit cardinal $\kappa'\geq k$. 
Let $X$ be a $\emptyset$-type-definable subset of $\C^\lambda$ (where $\lambda$ is a cardinal number). Let $F'$ be a closed, $\aut(\C')$-invariant equivalence relation defined on $S_X(\C')$, and $F$ a closed, $\aut(\C)$-invariant equivalence relation defined on $S_X(\C)$. We say that $F'$ and $F$ are {\em compatible} if $r[F']=F$, where $r\colon S_X(\C')\to S_X(\C)$ is the restriction map.

  Let $\fwap\subseteq S_X(\C)\times S_X(\C)$ be the finest closed, $\aut(\C)$-invariant equivalence relation on $S_X(\C)$ such that the flow $( \aut(\C), S_X(\C)/\fwap )$ is WAP, and let $\fta\subseteq S_X(\C)\times S_X(\C)$ be the finest closed, $\aut(\C)$-invariant equivalence relation on $S_X(\C)$ such that the flow $( \aut(\C), S_X(\C)/\fta )$ is tame. 

Here is the main result of this paper.

\begin{teor}\label{intro: compatible fwap and fta}
    The Ellis groups of the $\aut(\C)$-flows $S_X(\C)/\fwap$ and  $S_X(\C)/\fta$ (treated as groups with the $\tau$-topology) do not depend on the choice of $\C$ as long as $\C$ is $(\aleph_0+\lambda)^+$-saturated and strongly $(\aleph_0+\lambda)^+$-homogeneous.
    %The Ellis groups of $( \aut(\C), S(\C)/\fwap )$ and  $( \aut(\C), S(\C)/\fta )$ do not depend on the choice of the monster model $\C$.
\end{teor}

%This follows from the following more general result:
In fact, we prove the following general theorem.

\begin{teor}\label{compatible relations teor intro}
    Assume that $\C\prec \C'$ are $\aleph_0$-saturated and strongly $\aleph_0$-homogeneous. If $F'$ and $F$ are compatible equivalence relations respectively on $S_X(\C')$ and $S_X(\C)$, then the Ellis group of the flow $(\aut(\C'),S_X(\C')/F')$ is topologically isomorphic to the Ellis group of the flow  $(\aut(\C),S_X(\C)/F)$.
\end{teor}

%In order to deduce Theorem \ref{intro: compatible fwap and fta}, we prove the following

Theorem \ref{intro: compatible fwap and fta} is a consequence of Theorem \ref{compatible relations teor intro} and the following

\begin{teor}\label{theorem: compatibility for WAP and tame}
    Assume $\C \prec \C'$ are $(\aleph_0+\lambda)^+$-saturated and strongly $(\aleph_0+\lambda)^+$-homogeneous. Then $\fwapp$ is compatible with $\fwap$ and $\ftaa$ is compatible with $\fta$ (where $\fwapp$ and $\ftaa$ are the counterparts of $\fwap$ and $\fta$ computed for $\C')$.
\end{teor}

The main result of \cite{Krupiski2017BoundednessAA} about absoluteness of the Ellis group of a theory was reproved in \cite{hrushovski2022definability} using a new notion of {\em infinitary definablity patterns structures}. In Section \ref{section: idp}, we take the opportunity and present yet another approach and proof, using infinitary definablity patterns structures together with topological dynamics. This approach is inspired by Pierre Simon's seminar notes on \cite{hrushovski2022definability}. Various facts obtained in Section \ref{section: idp} are then used in the proofs of Theorems \ref{compatible relations teor intro} and \ref{theorem: compatibility for WAP and tame} which are included in Sections \ref{section: compatible quotients} and \ref{section: nip stable tame wap}, respectively. Moreover, the approach from Section \ref{section: idp} has a potential to be adapted to other contexts in the future (e.g., to Keisler measures or general topological dynamics).

In Section \ref{section: nip stable tame wap}, there is also another related (but much easier) application of our general Theorem \ref{compatible relations teor intro}. Namely, let $E^{\textrm{st}}_\emptyset$ [resp. $E^{\textrm{NIP}}_\emptyset$] be the finest $\emptyset$-type-definable equivalence relation on $X$ with stable [resp. NIP] quotient $X/E$. Let $\tilde{E}^{\textrm{st}}_\emptyset$ [resp. $\tilde{E}^{\textrm{NIP}}_\emptyset$] be the induced equivalence relation on $S_X(\C)$, as described after Proposition \ref{remark: pist and piNIP}. By $\tilde{E}'^{\textrm{st}}_\emptyset$ [resp. $\tilde{E}'^{\textrm{NIP}}_\emptyset$] we denote the analogously defined equivalence relation for the model $\C'$.

\begin{prop}
Without any assumption on saturation, 
the relation $\tilde{E}'^{\textrm{st}}_\emptyset$ is compatible with $\tilde{E}^{\textrm{st}}_\emptyset$, and  $\tilde{E}'^{\textrm{NIP}}_\emptyset$ is compatible with  $\tilde{E}^{\textrm{NIP}}_\emptyset$.
\end{prop}

\begin{cor}
    The Ellis group of $S_X(\C)/\tilde{E}^{\textrm{st}}_\emptyset$ (treated as a topological group with the $\tau$-topology) does not depend on the choice of $\C$ as long as $\C$ is at least $\aleph_0$-saturated and strongly $\aleph_0$-homogeneous.
\end{cor}

In the above discussion, we are talking about the finest closed, $\aut(\C)$-invariant equivalence relation $\fwap$ [resp. $\fta$] on the $\aut(\C)$-flow $S_X(\C)$ with  WAP [resp. tame] quotient, and about the finest $\emptyset$-type-definable equivalence relation $E^{\textrm{st}}_\emptyset$ [resp. $E^{\textrm{NIP}}_\emptyset$]  on $X$ with stable [resp. NIP] quotient. Why do they exist?  The existence of $\fwap$ [resp. $\fta$] is a folklore knowledge in topological dynamics, which we explain in Section \ref{section: topological dynamics}. The existence of $E^{\textrm{st}}_\emptyset$ [resp. $E^{\textrm{NIP}}_\emptyset$] follows from fact that stable [resp. NIP] quotients are closed under taking products. In the stable case, this property is well-known, but we include a proof for the reader's convenience. Regarding NIP, the observation seems to be new and more involved. Both proofs are included in Appendix B.

\section{Preliminaries}\label{section: preliminaries}

\subsection{Model theory}\label{subsection: model theory}

$T$ will usually denote a first order complete theory in a language $\mL$. A model $M$ of $T$ is {\em $\kappa$-saturated} if every complete type $p \in S_n(A)$ over a subset $A$ of $M$ of cardinality less than $\kappa$ has a realization in $M$; it is said to be {\em strongly $\kappa$-homogeneous} if every partial elementary map between subsets of $M$ of cardinality less than $\kappa$ extends to an isomorphism of $M$. By  a {\em monster model} $\C$ of $T$ one usually means a $\kappa$-saturated and strongly $\kappa$-homogeneous model for a sufficiently large cardinal $\kappa$ (often a strong limit cardinal greater than $|T|$). But in this paper we will usually require much less about $\kappa$. Precise requirements will be given in the main results. In any case, the above $\kappa$ is called {\em the degree of saturation of $\C$}.

Let $x$ be a tuple of variables of length $\lambda$ (typically, $\lambda < \kappa$). By a {\em $\emptyset$-type-definable set} we mean a partial type $\pi(x)$ without parameters in variables $x$ or the actual subset $X = \pi(\C)$ of $\C^{\lambda}$. By $S_\pi(\C)$ or $S_X(\C)$ we denote the space of complete types over $\C$ concentrated on $X$, that is containing $\pi(x)$ as a set of formulas.

The group $\aut(\C)$ of automorphisms of $\C$ acts naturally on the left on $S_X(\C)$ by $\sigma(p):=\{\varphi(x,\sigma(a)): \varphi(x,a) \in p\}$. Then $(\aut(\C),S_X(\C))$ is a flow with $\aut(\C)$ equipped with the pointwise convergence topology.

Let $\C' \succ \C$ be a bigger monster model whose degree of saturation is greater than $|\C|$. Then the complete types over $\C$ can be realized in $\C'$. For $p \in S(\C)$ and $a \in \C'$, $a \models p$ or $p = \tp(a/\C)$ means that $a$ realizes $p$.

Let $E$ be a $\emptyset$-type-definable equivalence relation on a $\emptyset$-type-definable subset $X=\pi(\C)$ of $\C^\lambda$, with $\lambda< \kappa$ (where $\kappa$ is the degree of saturation of $\C$). The equivalence classes of $E$ are called {\em hyperimaginary elements}, while all the sets of the form $X/E$ are called {\em hyperdefinable sets} (over $\emptyset$). 

Recall that the complete types over $\C$ of elements of $X'/E'$ (where $X'=\pi(\C')$ and $E':=E(\C')$)
can be defined as the $\aut(\C'/\C)$-orbits on $X'/E'$, or the preimages of these orbits
under the quotient map, or the partial types defining these preimages, or the classes
of the closed, $\aut(\C)$-invariant equivalence relation $\Tilde{E}$ on $S_X(\C)$ given by 
$$ p\Tilde{E}q\iff (\exists a\models p, b\models q) \left( aE'b \right).$$
By $\tp(\bslant{a/E}{\C})$ we denote the type of $a/E$ over $\C$.
In particular, 
$$ p\Tilde{E}q\iff (\exists a\models p, b\models q) \left( \tp(\bslant{a/E'}{\C})=\tp(\bslant{b/E'}{\C})  \right),$$
and
$$ p\Tilde{E}q\iff (\forall a\models p, b\models q) \left( \tp(\bslant{a/E'}{\C})=\tp(\bslant{b/E'}{\C})  \right).$$
The collection of all such types is denoted by $S_{X/E}(\C)$ and is equipped with the quotient topology on $S_X(\C)/\tilde{E}$. Then $(\aut(\C),S_{X/E}(\C))$ becomes an $\aut(\C)$-flow.

Complete types of hyperimaginaries over any sets of parameters (in place of the whole $\C$) are defined analogously. In particular, $\tilde{E}$ is defined for any (not necessarily sufficiently saturated) model $\C$ as long as $E$ is $\emptyset$-type-definable in the sense that there exists a partial type $\rho(x,y)$ over $\emptyset$ which defines an equivalence relation $E(M)$ on $X(M)$ in a sufficiently saturated (equivalently, in every) model $M$ of $T$.

Recall now what it means that $X/E$ is stable and NIP (see \cite{MR3796277, 10.1215/00294527-2022-0023}).

\begin{defin}\label{definition: stability}
	We say that $X/E$ is \emph{stable} if for every indiscernible sequence $(a_i,b_i)_{i<\omega}$ with $a_i\in X/E$ for all (equivalently, some) $i<\omega$, we have  
	$$\tp(a_i,b_j) = \tp(a_j,b_i)$$ for all (some) $i\neq j < \omega$.
\end{defin}

\begin{defin}\label{definition: NIP}
	We say that $X/E$ has {\em NIP} if there	do not exist an indiscernible sequence $(b_i)_{i<\omega}$ and $d\in X/E$ such that
	$((d, b_{2i}, b_{2i+1}))_{i<\omega}$ is indiscernible and $\tp(d, b_0) \neq \tp(d, b_1)$.
\end{defin}

Note that the $b_i$'s in the above definitions can be anywhere, not necessarily in $X/E$.

The properties of stability and NIP for hyperdefinable sets are both preserved under (possibly infinite) Cartesian products and taking type-definable subsets. Closedness of stability under taking products was remarked (without a proof) in \cite{MR3796277}; closedness of NIP under taking products seems to be a new result and the proof is more involved. Both proofs are included in Appendix \ref{appendixB}. As a corollary of these results, in Corollary \ref{finest eq exist for fixed param} we obtain the existence of the finest $\emptyset$-definable equivalence relation on $X$ with stable [resp. NIP] quotient. Below is a more precise information that partial types defining these finest equivalence relations do not depend on the choice of $\C$.

%%%Krzys:  $(\aleph_0 + \lambda)$-saturation is enough below, because any type in at most \kappa-variables can be realized in every \kappa-saturated model.

\begin{prop}\label{remark: pist and piNIP}
Given a partial type $\pi(x)$ over $\emptyset$ (with $|x|=\lambda$), there exists a partial type $\pi^{\textrm{st}}(x, y)$ [resp. $\pi^{\textrm{NIP}}(x, y)$] over the empty set which for every $(\aleph_0 + \lambda)$-saturated
model $\C$ defines the finest $\emptyset$-type-definable equivalence relation
on $X:=\pi(\C)$ with stable [resp. NIP] quotient. This relation will be denoted by $E^{\textrm{st}}_\emptyset$ [resp.  $E^{\textrm{NIP}}_\emptyset$].  
\end{prop}

\begin{proof}
    Let us fix a model $\C$ and $X\subseteq \C^\lambda$ as above. Corollary \ref{finest eq exist for fixed param} implies that a finest $\emptyset$-type-definable equivalence relation on $X$ with stable [resp. NIP] quotient exists. Let $\pi^{\textrm{st}}$ [resp. $\pi^{\textrm{NIP}}$] be the partial type defining this relation. Put $X' :=\pi(\C')\subseteq \C'^{\lambda}$ where $\C'$ is a different $(\aleph_0 +\lambda)$-saturated model. It is a routine exercise to check that $\pi^{\textrm{st}}(\C',\C')$ [resp. $\pi^{\textrm{NIP}}(\C',\C')$] defines the finest $\emptyset$-type-definable equivalence relation on $X'$ with stable [resp. NIP] quotient.
\end{proof}

Even if $\C$ is not sufficiently saturated,
then by $E^{\textrm{st}}_\emptyset$ we could mean  $\pi^{\textrm{st}}(X(\C),X(\C))$. However, we do not need to talk about
it, as we will only work with the equivalence relation $\tilde{E}^{\textrm{st}}_\emptyset$ on $S_X(\C)$ which is defined as we did above. Namely,
$$p \tilde{E}^{\textrm{st}}_\emptyset q \iff (\exists a \models p, b \models q)( \pi^{\textrm{st}}(a,b)),$$
where $a,b$ are taken in a big monster model.
Similarly for NIP, we are interested in the relation $\tilde{E}^{\textrm{NIP}}_\emptyset$ on $S_X(\C)$ defined by
$$p \tilde{E}^{\textrm{NIP}}_\emptyset q \iff (\exists a \models p, b \models q)( \pi^{\textrm{NIP}}(a,b)).$$
By  $E'^{\textrm{st}}_\emptyset$, $E'^{\textrm{NIP}}_\emptyset$, $\tilde{E}'^{\textrm{st}}_\emptyset$, and $\tilde{E}'^{\textrm{NIP}}_\emptyset$ we denote the relations defined as above but working with $\C'$ in place of $\C$ (where $\C'$ is another model of $T$).

\subsection{Topological dynamics}\label{section: topological dynamics}

We recall various definitions and state some facts concerning topological dynamics. For a more in depth study of the topic see e.g. \cite{auslander1988minimal} and \cite{Glasner1976}. A concise presentation (with proofs) of basic Ellis theory can be found in Appendix A of \cite{rzepecki2018bounded}.

In this paper, compact spaces are Hausdorff by definition.

\begin{defin}\label{defin: ellis semigroup}
\begin{itemize}
    \item A \emph{$G$-flow} is a pair $(G,X)$ consisting of a topological group $G$ that acts continuously on a compact space $X$.
    \item If $(G,X)$ is a $G$-flow, then its \emph{Ellis semigroup} $E(X)$ is the pointwise closure in $X^X$ of the set of functions $\pi_g: x \mapsto g\cdot x$ for $g\in G$.
\end{itemize}
\end{defin}

\begin{fact}
    The Ellis semigroup of a $G$-flow $(G,X)$ is a compact left topological semigroup with composition as its semigroup operation. Moreover, $E(X)$ is itself a $G$-flow equipped with the action $g\eta:=\pi_g \circ \eta$ for $g\in G$ and $\eta \in E(X)$.
\end{fact}

Recall that a left ideal $I$ of a semigroup $S$, written as $I\unlhd S$, is a subset of $S$ such that $SI\subseteq I$.

The following is due to Ellis \cite{Ell69}. For a proof see also \cite[Fact A.8]{rzepecki2018bounded}.

\begin{fact}\label{Ellis theorem}
Let $(G,X)$ be a flow.
 Minimal left ideals of $E(X)$ exist and coincide with the minimal subflows of $(G,E(X))$. If $\mathcal{M}\unlhd E(X)$ is a minimal left ideal, then:
\begin{itemize}
    \item The ideal $\mathcal{M}$ is closed and for every $a\in \mathcal{M}$ we have $\mathcal{M}=E(X)a$.
    \item The set of idempotents of $\mathcal{M}$, denoted by $\mathcal{J}(\mathcal{M})$, is nonempty. Moreover, $\mathcal{M}=\bigsqcup_{u\in\mathcal{J}(\mathcal{M})} u\mathcal{M}$.
    \item For every $u\in\mathcal{J}(\mathcal{M})$, $u\mathcal{M}$ is a group with the neutral element $u$. Moreover, the isomorphism type of this group does not depend on the choice of $u$ and $\mathcal{M}$. We will call it the $\emph{Ellis group}$ of $X$. (This is the terminology used by model theorists; in the topological dynamics literature ``Ellis group'' usually denotes a related but different object.)
\item Let $u \in\mathcal{J}(\mathcal{M})$. For any minimal left ideal $\mathcal{N}$ there exists an idempotent $v \in \mathcal{N}$ such that $vu=u$ and $uv=v$.
    \item For every $u\in\mathcal{J}(\mathcal{M})$ and $s\in \mathcal{M}$, $su=s$.
%    \item For every minimal left ideal $\mathcal{N}\unlhd E(X)$ and for every $u\in \mathcal{J}(\mathcal{M})$, $v\in \mathcal{J}(\mathcal{N})$ we have $u\mathcal{M}\cong v\mathcal{N}$.
\end{itemize} 
\end{fact}
  
Let $(G,X)$ be a flow, $\mathcal{M} \unlhd E(X)$ a minimal lef ideal, and $u \in  \mathcal{J}(\mathcal{M})$. 
The Ellis group $u\mathcal{M}$ has the inherited topology from $E(X)$. However, there exists another important topology on $u\mathcal{M}$, called the \emph{$\tau$-topology}. We recall it now (for the proofs see \cite[Appendix A]{rzepecki2018bounded}). First, for any $a\in E(X)   $ and $B\subseteq E(X)$ let $a\circ B$ be the set of all limits of the nets $(g_ib_i)_{i\in \I}$ such that $g_i\in G$, $b_i\in B$ and $\lim g_i=a$ (where $g_i$ is identified with $\pi_{g_i}$). The closure operator $\cl_\tau$ on subsets of $u\mathcal{M}$ is given by $\cl_\tau(B):=u\mathcal{M}\cap (u\circ B) = u(u\circ B)$ (for $B\subseteq u\mathcal{M}$). The {\em $\tau$-topology} is the topology induced on $u\mathcal{M}$ by the closure operator $\cl_\tau$. The Ellis group $u\mathcal{M}$ equipped with with the $\tau$-topology is a quasi-compact $T_1$ semitopological group (i.e. the group operation is separately continuous). In fact, in the third item of Fact \ref{Ellis theorem}, we have that the isomorphism type of $u\mathcal{M}$ as a group equipped with the $\tau$-topology does not depend on the choice of $u$ and $\mathcal{M}$ (for a proof see e.g. \cite[Fact A.37]{rzepecki2018bounded}).

For a proof of the following fact see \cite[Proposition 5.41]{rzepecki2018bounded}.

\begin{fact}\label{epimorphism semigroup to group}
    Let $(G,X)$ and $(G,Y)$ be two $G$-flows, and let $\Phi: X\to Y$ be a $G$-flow epimorphism. Then $\Phi_*:E(X)\to E(Y)$ given by $$\Phi_*(\eta)(\Phi(x)):=\Phi(\eta(x))$$ is a continuous epimorphism.  
    
    If $\mathcal{M}$ is a minimal left ideal of $E(X)$ and $u\in \mathcal{J}(\mathcal{M})$, then:
    \begin{itemize}
        \item $\mathcal{M}':=\Phi_*[\mathcal{M}]$ is a minimal left ideal of $E(Y)$ and $u'=\Phi_*(u)\in \mathcal{J}(\mathcal{M}')$.
        \item $\Phi_*\!\!\upharpoonright_{u\mathcal{M}}: u\mathcal{M}\to u'\mathcal{M}'$ is a group epimorphism and a quotient map in the $\tau$-topologies.
    \end{itemize}
    Moreover, if $\Phi: X\to Y$ is a $G$-flow isomorphism, then $\Phi_*\!\!\upharpoonright_{u\mathcal{M}}: u\mathcal{M}\to \Phi(u)\Phi[\mathcal{M}]$ is a group isomorphism and a homeomorphism in the $\tau$-topologies.
\end{fact}

From the last item of Fact \ref{Ellis theorem}, we easily deduce the following:

\begin{remark}\label{restrictio to image is monomorphism}
    If $X$ is a $G$-flow, $\mathcal{M}$ a minimal left ideal in $E(X)$, and $u \in \mathcal{M}$ an idempotent, then the map $f \colon u\mathcal{M} \to \Sym(\Image(u))$ given by $f(\eta):=\eta\!\!\upharpoonright_{\Image(u)}$ is a group monomorphism.
\end{remark}

We now briefly discuss Ellis semigroups in a model-theoretic context. In particular, we recall the definition of \emph{content}, introduced in \cite[Definition 3.1]{Krupiski2017BoundednessAA}, which will be an important tool in this paper.

%%%Krzys: I added below the requirement that $y$ is finite. Similarly in Def. 2.27 in the next subsection. This is not needed working in a single model, e.g. Fact 2.10 is true in both versions. But the proof of Lemma 2.40 requires $y$ to be finite. I think this was the situation we had in mind in [KNS19], but it was not explicitly written.
\begin{defin}\label{defin: content} Fix $A\subseteq B$.
\begin{itemize}
    \item For $p(x)\in S(B)$, the \emph{content} of $p$ over $A$ is the following set:
    $$ c_A(p):= \{ (\varphi(x,y),q(y))\in \mL(A)\times S_y(A) : y \text{ is finite and } \varphi(x,b)\in p(x) \text{ for some } b\models q  \} .$$
    \item Similarly, the {\em content} of a sequence $p_0(x),\dots,p_n(x)\in S(B)$ over $A$, $c_A( p_0,\dots,p_n )$, is defined as the set of all $ (\varphi_0(x,y),\dots ,\varphi_n(x,y),q(y))\in \mL(A)^n\times S_y(A)$ with $y$ finite such that for some $b\models q$ and for every $i\leq n$ we have $\varphi_i(x,b)\in p_i$.
\end{itemize}
    If $A=\emptyset$, we simply omit it.
\end{defin}

The fundamental connection between contents and the Ellis semigroup of the flow $(\aut(\C),S_\pi(\C))$ is the following. For a proof see \cite[Proposition 3.5]{Krupiski2017BoundednessAA}. In fact, in this result it is enough to assume that $\C$ is strongly $\aleph_0$-homogeneous.

\begin{fact}\label{content and ellis semigroup}
    Let $\pi(x)$ be a partial type over $\emptyset$, $S_\pi(\C)$ the set of complete types over $\C$ extending $\pi$, and $(p_0,\dots,p_n)$ and $(q_0,\dots,q_n)$ sequences from $S_{\pi}(\C)$. Then $c(q_0,\dots,q_n)\subseteq c(p_0,\dots, p_n)$ if and only if there exists $\eta\in E(S_\pi(\C))$ such that $\eta(p_i)=q_i$ for every $i\leq n$.
\end{fact}
%The proof of this fact can be found in \cite[Proposition 3.5]{Krupiski2017BoundednessAA}.

\begin{defin}\label{defin: content order}
    Let $\pi(x)$ be a partial type over $\emptyset$ and $(p_0,\dots,p_n)$ and $(q_0,\dots,q_n)$ sequences from $S_{\pi}(\C)$. We write $(q_0,\dots,q_n)\leq^c (p_0,\dots, p_n)$ if $c(q_0,\dots,q_n)\subseteq c(p_0,\dots, p_n)$.
\end{defin}

For the rest of the section we fix a $G$-flow $(G,X)$. Let $C(X)$ denote the space of all continuous real-valued maps on $X$. Given $f\in C(X)$ and $g\in G$, we define $gf \in C(X)$ by $(gf)(x):=f(g^{-1} x)$. This is a left action of $G$ on $C(X)$.

We recall two important classes of flows: \emph{weakly almost periodic flows} and \emph{tame flows}. For a more in depth treatment of the topic we recommend \cite{55bad191-0781-3b5e-a214-06bd2e8fc797} for weakly almost periodic flows and  \cite{Glasner2018} for tame ones.

Recall that the \emph{weak topology} on $C(X)$ is defined as the coarsest topology such that for every bounded (equivalently, continuous with respect to the supremum norm) linear functional $\ell:C(X)\to \R$, the map $$\ell:C(X)\to \R$$
is continuous.

\begin{defin}
    We say that a function $f\in C(X)$ is \emph{weakly almost periodic} ({\em WAP}) if $(gf:g\in G)$ is relatively compact in the weak topology on $C(X)$. A flow $(G,X)$ is {\em WAP} if every $f\in C(X)$ is WAP.
\end{defin}

The following fact is due to Grothendieck \cite{f935e45f-f66d-33ab-9bf8-2d48a5e00aee} (see also \cite[Appendix D]{kerr2016ergodic}).

\begin{fact} \label{double limit}
    Let $X_0$ be any dense subset of $X$. Let $f\in C(X)$. The following are equivalent:
    \begin{itemize}
        \item $f$ is WAP.
        \item  $\{gf:g\in G\}$ is relatively compact in the topology of pointwise convergence on $C(X)$.
\item The pointwise closure of $\{gf:g\in G\}$ is contained in $C(X)$.
        \item For any sequences $(g_nf)_{n<\omega}\subseteq \{gf:g\in G\}$ and $(x_n)_{n<\omega}\subseteq X_0$ we have $$\lim_n\lim_m g_nf(x_m)=\lim_m\lim_n g_nf(x_m) $$ whenever both limits exits.
    \end{itemize}
\end{fact}

The next two facts will be useful in this paper:

\begin{fact}\label{wap closed unital}
    For any flow $(G,X)$, the WAP functions form a closed (with the supremum norm) unital subalgebra of $C(X)$. 
\end{fact}

% Details about this result can be found in \cite{ibarlucia2016dynamical} after Fact 2.1. 
The above fact is well-known and can be easily shown using Fact \ref{double limit}.
Combining it with the Stone-Weierstrass theorem, we obtain the second fact:

\begin{fact} \label{separating points: WAP}
    If $\mathcal{A}\subseteq C(X)$ is a family of functions that separate points, then $(G,X)$ is WAP if and only if every $f\in \mathcal{A}$ is WAP.
\end{fact}

\begin{defin}\label{defin: indep functons}
    We say that a sequence of functions $(f_n)_{n<\omega} \in C(X)$ is \emph{independent} if there are real numbers $r<s$ such that $$ \bigcap_{n\in P}f_n^{-1}(-\infty,r) \cap \bigcap_{n\in M}f_n^{-1}(s,\infty)\neq \emptyset$$ for all finite disjoint $P,M\subset\omega$. Given a dense $X_0\subseteq X$, we can equivalently require $$ \bigcap_{n\in P}f_n^{-1}(-\infty,r) \cap \bigcap_{n\in M}f_n^{-1}(s,\infty)\cap X_0\neq \emptyset$$ for all finite disjoint $P,M\subset\omega$.
\end{defin}

\begin{defin}
    Let $\{ f_n: X\to \R \}_{n\in \mathbb{N}}$ be a uniformly bounded sequence of functions. We say that this sequence is an $\ell_1$-sequence on $X$ if there exists a real constant $a>0$ such that for all $n\in \mathbb{N}$ and choices of real numbers $c_1,\dots,c_n$ we have $$ a\cdot \sum_{i=1}^n \lvert c_i\lvert \leq \left\lVert \sum_{i=1}^n c_i f_i \right\rVert_\infty.$$
\end{defin}

The following equivalences can be found in \cite[Theorem 2.4]{Glasner2018}

\begin{fact}\label{l_1 and indep seq}
    The following are equivalent for a bounded $F\subseteq C(X)$:
\begin{itemize}
    \item $F$ does not contain an independent sequence.
    \item  $F$ does not contain an $\ell_1$-sequence.
    \item Each sequence in $F$ has a pointwise convergent subsequence in $\R^X$.
\end{itemize}
\end{fact}

\begin{defin}
    We say that a function $f\in C(X)$ is \emph{tame} if $\{gf:g\in G\}$ does not contain an independent sequence. A flow $(G,X)$ is {\em tame} if every $f\in C(X)$ is tame.
\end{defin}

The next two facts will be useful throughout this paper:

\begin{fact}\label{tame closed unital}
    For any flow $(G,X)$, the tame functions form a closed (with the supremum norm) unital subalgebra of $C(X)$. 
\end{fact}
A proof can be found e.g. in \cite[Fact 2.72]{rzepecki2018bounded}. From this fact and Stone-Weierstrass theorem, we obtain the following:

\begin{fact} \label{separating points: tame}
    If $\mathcal{A}\subseteq C(X)$ is a family of functions that separate points, then $(G,X)$ is tame if and only if every $f\in \mathcal{A}$ is tame.
\end{fact}

By compactness, every epimorphism $\rho\colon X \to Y$ of $G$-flows is automatically a topological quotient map (meaning that a subset of $Y$ is open if and only if its preimage under $\rho$ is open) and can be identified with the quotient map $X \to X/\!\sim$, where $\sim$ is the closed, invariant equivalence relation on $X$ given by $x \sim y \iff \rho(x)=\rho(y)$. Conversely, every closed, invariant equivalence relation $\sim$ on $X$ yields the quotient $G$-flow $(G,X/\!\sim)$ and the quotient epimorphism $X \to X/\!\sim$.

We already defined a left action of $G$ on $C(X)$: $(gf)(x):=f(g^{-1}x)$. By a unital, closed $G$-subalgebra of $C(X)$ we mean a closed (in the supremuem norm topology) subalgbra of $C(X)$ which is closed under the left action of $G$ and contains the constant function equal to $1$. The next fact belongs to folklore and an easy proof is left as a routine exercise using Stone-Weierstrass theorem.

\begin{fact}\label{fact: relations vs algebras}
Let $(G,X)$ be a flow. There is a 1-1 correspondence between closed, invariant equivalence relations on $X$ and closed, unital  $G$-subalgebras of $C(X)$. 

More precisely, for any closed, invariant equivalence relation $\sim$ on $X$, let $Y:=X/\!\sim$ and $\pi \colon X \to Y$ be the quotient map, and put
$$\Phi(\sim):= \{f \circ \pi: f \in C(Y)\}.$$
%Then $\Phi(\sim)$ is a closed, unital subalgebra of $C(X)$.
Conversely, for any closed, unital $G$-subalgebra $\mathcal{A}$ of $C(X)$ put
$$\Psi(\mathcal{A}):= \; \sim,$$
where $x \sim y$ if and only if $f(x)=f(y)$ for all $f \in \mathcal{A}$. 
%Then $\sim$ is a closed, invariant equivalence relation on $X$.
Then $\Phi$ is a bijection from the set of closed, invariant equivalence relations on $X$ to the set of closed, unital $G$-subalgebras of $C(X)$, and $\Psi$ is the inverse of $\Phi$.
\end{fact}

\begin{remark}\label{remark: transfer of WAP and tame}
Let  $\rho\colon X \to Y$ be an epimorphism of $G$-flows. Let $h \in C(Y)$ and $f:= h\rho$. Then $f \in C(X)$ and:
\begin{enumerate}
\item $f$ is WAP if and only if $h$ is WAP.
\item $f$ is tame if and only if $h$ is tame.
\end{enumerate}
\end{remark}

\begin{proof}
Item (1) follows immediately from the characterization of WAP via the double limit criterion from Fact \ref{double limit}. Item (2) follow directly from the definition of tame flows.
\end{proof}

\begin{cor}\label{corollary: finest with WAP/tame quotient}
Let $(G,X)$ be a flow.
\begin{enumerate}
\item There exists a finest closed, invariant equivalence relation $F_{\textrm{WAP}}$ on $X$ such that the quotient flow $(G,X/F_{\textrm{WAP}})$ is WAP.
\item There exists a finest closed, invariant equivalence relation $F_{\textrm{tame}}$ on $X$ such that the quotient flow $(G,X/F_{\textrm{tame}})$ is tame.
\end{enumerate}
\end{cor}

\begin{proof}
Let $\textrm{WAP}(X)\subseteq C(X)$ be the set of all WAP functions on $X$, and $\textrm{Tame}(X)\subseteq C(X)$ the set of all tame functions on $X$. By Facts \ref{wap closed unital} and \ref{tame closed unital}, we know that both $\textrm{WAP}(X)$ and $\textrm{Tame}(X)$ are closed, unital subalgebras of $C(X)$. And it is clear that they are $G$-subalgebras. By Fact \ref{fact: relations vs algebras} and Remark \ref{remark: transfer of WAP and tame}, the corresponding equivalence relations $F_{\textrm{WAP}}:=\Psi(\textrm{WAP}(X))$ and $F_{\textrm{tame}}:=\Psi(\textrm{Tame}(X))$ fulfill the requirements.
\end{proof}

We finally apply these general considerations to our model-theoretic context. So $\C$ is a monster model of $T$ (although here it could be actually any model of $T$), and $X$ is a $\emptyset$-type-definable set (so here $X$ does not denote a flow anymore). As explained in Section \ref{subsection: model theory}, we have the flow $(\aut(\C),S_X(\C))$. By Corollary \ref{corollary: finest with WAP/tame quotient}, there exists a finest closed, $\aut(\C)$-invariant equivalence relation $\fwap$ on $S_X(\C)$ such that the flow $( \aut(\C), S_X(\C)/\fwap )$ is WAP. Similarly, there exists a finest closed, $\aut(\C)$-invariant equivalence relation $\fta$ on $S_X(\C)$ such that the flow $( \aut(\C), S_X(\C)/\fta )$ is tame. Whenever we work with another model denoted by $\C'$, the corresponding equivalence relations will be denoted by $\fwapp$ and $\ftaa$, respectively.

\subsection{Infinitary definability patterns}\label{section: idp}

We introduce the necessary machinery of infinitary definability patterns structure on $S_X(\C)$ that will be used throughout the rest of the paper. The results here are based on the first author's course on topological dynamics in model theory, which is an alternative approach to Hrushovski's ``infinitary core" inspired by Pierre Simon's seminar notes on \cite{hrushovski2022definability}. This approach is a combination of Simon's approach to Hrushovski's definability patterns and topological dynamics, and has a potential to wide generalizations, e.g. to Keisler measures or more general topological dynamics contexts, which will be studied in the future. The notion of ip-minimal set below, some basic lemmas about it, and an application to get Theorem \ref{absoluteness of the core} are adapted from Simon's notes.

We will assume in this section that $\C$ is at least $\aleph_0$-saturated and strongly $\aleph_0$-homogeneous. Let $X$ be a $\emptyset$-type-definable set of tuples of arbitrary (possibly infinite) length $\lambda$ and $x$ be a corresponding tuple of variables (i.e., of length $\lambda$). All sets of parameters are contained in $\C$.

The notions of content of a tuple of types $\bar p$ (denoted by $c(\bar p)$) and of the order $\leq^c$ were already recalled in Definitions \ref{defin: content} and \ref{defin: content order}. Now, we recall  the notion of strong heir, which first appeared in \cite[Definition 3.2]{Krupiski2017BoundednessAA}.
%%%Krzys:  \overline{b}\subseteq B in place of  \overline{b}\in B. And I removed \overline.  Also, ``extension'' was missing.
\begin{defin}\label{strong heirs}
	Let $M \prec \C$ be contained in $B$ and $q(x)\in S(B)$.
	We say that $q(x)$ is a \emph{strong heir extension} of $q\!\upharpoonright_M(x)$ or that it is a {\em strong heir} over $M$ if for every finite $m\subseteq M$ and finite tuple of variables $y$
	$$(\forall \varphi(x,y)\in \mL)(\forall b\subseteq B)[ \varphi(x,b)\in q(x) \implies (\exists b'\subseteq M)(\varphi(x,b')\in q(x)\wedge b\underset{m}{\equiv}b') ].$$
\end{defin}

The next fact is \cite[Lemma 3.3]{Krupiski2017BoundednessAA}.

\begin{fact}\label{fact: existence of heirs}
Assume $M \subseteq A \subseteq \C$, where $M\prec \C$ is $\aleph_0$-saturated. Then each type $p(x) \in S(M)$ (in possibly infinitely many variables $x$) has an extension $p'(x) \in S(A)$ which is a strong heir over $M$.
\end{fact}

The next definition is due to Hrushovski \cite{hrushovski2022definability}.

\begin{defin}[Infinitary definability patterns structure on $S_X(\C)$]
For any $n$-tuple $\overline{\varphi} = (\varphi_1(x,y),\dots,\varphi_n(x,y))$ of formulas  in $\mL$ with $y$ finite and $q(y)\in S_y(\emptyset)$, we define $R_{\overline{\varphi},q}$ on $S_X(\C)^n$ by $$R_{\overline{\varphi},q}(\overline{p}) \iff (\varphi_1(x,y),\dots,\varphi_n(x,y),q)\notin c(\overline{p}), $$ where $\overline{p}=(p_1,\dots,p_n)$,
    i.e.,  there is no  $b\models q$  such that  $\varphi_1(x,b)\in p_1\wedge\cdots\wedge \varphi_n(x,b)\in p_n$. The \emph{infinitary definability patterns structure on} $S_X(\C)$ consists of all such relations $R_{\overline{\varphi},q}$.  We denote by $\End(S_X(\C))$ the semigroup of endomorphisms of $S_X(\C)$ as the infinitary definability patterns structure.
\end{defin}

In this section, we also consider $S_X(\C)$ as the flow $(\aut(\C),S_X(\C)).$ Recall that by $E(S_X(\C))$ we denote the Ellis semigroup of this flow (see Definition \ref{defin: ellis semigroup}).

\begin{lema}\label{S_X is homogeneous}
    We have the following:
    \begin{itemize}
        \item $\End(S_X(\C))=E(S_X(\C))$.
        \item $S_X(\C)$ is homogeneous in the sense that any partial morphism between substructures of $S_X(\C)$ (i.e., any structure preserving map $f:A\to B$ where $A,B\subseteq S_X(\C)$) extends to an endomorphism of $S_X(\C)$.
    \end{itemize}
\end{lema}
\begin{proof}
    $\End(S_X(\C))\supseteq E(S_X(\C))$ follows from the right to left implication of Fact \ref{content and ellis semigroup}.

    The inclusion $\End(S_X(\C))\subseteq E(S_X(\C))$ and homogeneity follow from the left to right implication of Fact \ref{content and ellis semigroup} and compactness of $E(S_X(\C))$.
\end{proof}

\begin{prop}\label{group isomorphism delta}
    Let $\mathcal{M}\unlhd E(S_X(\C))$ be a minimal left ideal and $u\in \mathcal{M}$ an idempotent. Let $\overline{\mathcal{J}}:=\Image(u)\subseteq S_X(\C)$. Then the map $$\delta: u\mathcal{M}\to \aut(\overline{\mathcal{J}})$$ given by $\delta(\eta):=\eta\!\!\upharpoonright_{\overline{\mathcal{J}}}$ is a group isomorphism, where $\aut(\overline{\mathcal{J}})$ is the group of automorphisms of $\overline{\mathcal{J}}$ as the infinitary definability patterns structure.
\end{prop}

\begin{proof}
    %From Fact \ref{content and ellis semigroup} it follows that $\delta$ take values in $\aut(\overline{\mathcal{J}})$
%   Since $u\mathcal{M}$ is a group, we easily get that $\delta$ takes values in $\Sym(\overline{\mathcal{J}})$. The fact that $\delta$ takes values in $\aut(\overline{\mathcal{J}})$ follows from Lemma \ref{S_X is homogeneous}. Clearly, the map $\delta$ is a homomorphism. Injectivity follows from the fact that $u\mathcal{M}u=u\mathcal{M}$ and surjectivity follows by homogeneity of $S_X(\C)$.
By Remark \ref{restrictio to image is monomorphism}, $\delta$ is a monomorphism from $u\mathcal{M}$ to $\Sym(\overline{\mathcal{J}})$. The fact that $\delta$ takes values in $\aut(\overline{\mathcal{J}})$ follows from the first part of Lemma \ref{S_X is homogeneous}. The fact that $\delta$ is onto $\aut(\overline{\mathcal{J}})$ follows by homogeneity of $S_X(\C)$.
\end{proof}

\begin{prop}\label{isomorphic image idempotent}
    For any minimal left ideals $\mathcal{M}, \mathcal{N}$ of $E(S_X(\C))$ and idempotents $u\in \mathcal{M}$ and $v\in \mathcal{N}$, $\Image(u)\cong\Image(v)$ as the infinitary definability patterns structures.
\end{prop}

    \begin{proof}
        By Fact \ref{Ellis theorem}, there is an idempotent $u'\in \mathcal{M}$ such that $vu'=u'$ and $u'v=v$. Then, $\Image(u')=\Image(v)$, so we can assume that $\mathcal{M}=\mathcal{N}$ without loss of generality. Then, $uv=u$ and $vu=v$, and so the maps\begin{align*}
            u\!\!\upharpoonright_{\Image(v)}: &\Image(v)\to \Image (u)\\
            &\text{ and }\\
            v\!\!\upharpoonright_{\Image(u)}: &\Image(u)\to \Image (v)
        \end{align*} are mutual inverses. Hence, $\Image(v)\cong \Image (u)$ by Lemma \ref{S_X is homogeneous}.
    \end{proof}

By Proposition \ref{isomorphic image idempotent}, up to isomorphism, both $\overline{\mathcal{J}}=\Image(u)$ and $\aut(\overline{\mathcal{J}})$ do not depend on the choice of the minimal left ideal $\mathcal{M}$ and idempotent $u\in\mathcal{M}$.

The following definition is an analog of Hrushovski's definition of pp-topology for definability patterns \cite{hrushovski2022definability}.

\begin{defin}[ipp-topology]\label{definition: ipp-topology}
    The \emph{ipp-topology} on $\aut(\overline{\mathcal{J}})$ is given by the subbasis of closed sets consisting of $$F_{\bar{\varphi}, \bar{p},\bar{q}, r}:=\{f\in \aut(\overline{\mathcal{J}}): R_{\overline{\varphi},r}(f(p_1),\dots, f(p_m), q_{1},\dots, q_n ) \} $$ for any $\varphi_1(x,y),\dots,\varphi_{m+n}(x,y)\in \mL$, $r\in S_y(\emptyset)$, and $p_1,\dots,p_m,q_{1},\dots,q_n\in \overline{\mathcal{J}} $.
\end{defin}

One can easily check that the above subbasis is in fact a basis of closed sets, i.e. the union of any two sets from this subbasis is the intersection of a family of subbasic sets, but we will not need it.

The proof of the next proposition is a bit technical, so we move it to Appendix \ref{Appendix}.

\begin{prop}\label{delta is homeomorphism}
    The map  $$\delta: u\mathcal{M}\to \aut(\overline{\mathcal{J}}) $$ from Proposition \ref{group isomorphism delta}  is a homeomorphism when $u\mathcal{M}$ is equipped with the $\tau$-topology and $\aut(\overline{\mathcal{J}})$ with the ipp-topology.
\end{prop}

\begin{defin}
    A subset $Q\subseteq S_X(\C)$ is \emph{ip-minimal} (from {\em infinitary patterns minimal}) if any morphism $f:Q\to S_X(\C)$ is an isomorphism onto $\Image(f)$.
\end{defin}

%%%Krzys: I removed the old item  (3), as $\eta_0$ is trivially equal to $\eta$.
\begin{prop}\label{equivalences ip minimal}
    Let $\overline{p}$ be an enumeration of $S_X(\C)$ and $\overline{q}=\eta \overline{p}$ (coordinate-wise) for some $\eta\in E(S_X(\C))$. Then the following are equivalent:
    \begin{enumerate}
        \item $\overline{q}$ is $\leq^c$ minimal in $E(S_X(\C))\overline{p}$, where $\overline{q}'\leq^c \overline{q}''$ means $$(q'_{i_1},\dots q'_{i_n})\leq^c (q''_{i_1},\dots q''_{i_n})$$ for any finite sets of indices $i_1<\dots<i_n$, or equivalently, $$ \overline{q}'\leq^c \overline{q}''\iff \overline{q}'\in E(S_X(\C))\overline{q}''  .$$
        \item The coordinates of $\overline{q}$ form an ip-minimal subset $Q$.
%        \item There is a minimal left ideal $\mathcal{M}\unlhd E(S_X(\C))$  and $\eta_0\in\mathcal{M}$ such that $\overline{q}=\eta_0\overline{p}$.
        \item $\eta$ belongs to a minimal left ideal of $E(S_X(\C))$.
    \end{enumerate}
\end{prop}
\begin{proof}
    $(1)\Rightarrow (2)$. Take any $f:Q\to S_X(\C)$. Lemma \ref{S_X is homogeneous} implies that $f$ can be extended to an endomorphism $\tilde{f}:S_X(\C)\to S_X(\C)$, and then $\eta:=\tilde{f} \in E(S_X(\C))$. Hence, $\eta\overline{q}\leq^c \overline{q}$. By minimality of $\overline{q}$, there is $\eta'\in E(S_X(\C))$ such that $\eta'\eta \overline{q}=\overline{q}$. Thus, $f$ is an isomorphism to its image by Lemma \ref{S_X is homogeneous}.

    $(2)\Rightarrow (1)$. Take any $\overline{q}'\in E(S_X(\C))\overline{p}$ such that $\overline{q}'\leq^c \overline{q}$, that is, $\overline{q}'=\eta \overline{q}$ for some $\eta\in E(S_X(\C))$. By Lemma \ref{S_X is homogeneous}, $f:=\eta\!\! \upharpoonright_Q: Q\to S_X(\C)$ is a morphism, hence it is an isomorphism to its image and has an inverse $f': \Image(f)\to Q$. By Lemma \ref{S_X is homogeneous}, $f'$ can be extended to $\eta'\in E(S_X(\C))$. Then $\eta'\overline{q}'=\overline{q}$. This implies that $\overline{q}\leq^c\overline{q}'$.

    $(1)\Rightarrow (3)$ Consider any $\eta' \in E(S_X(\C))$. Then, $\eta'\overline{q}\leq^c \overline{q}$ and $\eta'\overline{q}=\eta'\eta\overline{p}$. By $(1)$, $\overline{q}\leq^c\eta'\overline{q}$, so there is $\eta''\in E(S_X(\C))$ such that $\eta'' \eta' \overline{q}=\overline{q}$, that is $\eta''\eta'\eta \overline{p}=\eta\overline{p}$. Since $\overline{p}$ is an enumeration of $S_X(\C)$, we get $\eta''\eta'\eta=\eta$ as elements of $E(S_X(\C))$. Hence, $E(S_X(\C))\eta$ is a minimal left ideal.

%    $(4)\Rightarrow(3)$. Trivial.

 %   $(3)\Rightarrow(1)$. Let $\overline{q}=\eta_0\overline{p}$, where $\eta_0$ belongs to a minimal left ideal $\mathcal{M}$. Consider any $\overline{q}'\in E(S_X(\C))$ such that $\overline{q}'\leq^c\overline{q}$. Then, $$\overline{q}'=\eta'\overline{q}=\eta'\eta_0\overline{p}$$ for some $\eta'\in E(S_X(\C))$. Since $E(S_X(\C))\eta_0$ is a minimal left ideal, there is $\eta''\in E(S_X(\C))$ such that $\eta''\eta'\eta_0=\eta_0$. Thus, $\eta''\overline{q}'=\eta''\eta'\eta_0\overline{p}=\eta_0\overline{p}=\overline{q}$. Therefore, $\overline{q}\leq^c\overline{q}'$.
$(3)\Rightarrow(1)$. Consider any $\overline{q}'\in E(S_X(\C))\overline{p}$ such that $\overline{q}'\leq^c\overline{q}$. Then, $$\overline{q}'=\eta'\overline{q}=\eta'\eta\overline{p}$$ for some $\eta'\in E(S_X(\C))$. Since $E(S_X(\C))\eta$ is a minimal left ideal, there is $\eta''\in E(S_X(\C))$ such that $\eta''\eta'\eta=\eta$. Thus, $\eta''\overline{q}'=\eta''\eta'\eta\overline{p}=\eta\overline{p}=\overline{q}$. Therefore, $\overline{q}\leq^c\overline{q}'$.
\end{proof}

\begin{cor} \label{ip-minimal exists}
    There exists an ip-minimal $Q\subseteq S_X(\C)$ with a morphism $$g:S_X(\C)\to Q.$$
\end{cor}

\begin{proof}
    Let $\overline{p}$ be an enumeration of $S_X(\C)$, $\mathcal{M}\unlhd E(S_X(\C)))$ a minimal left ideal and $\eta \in \mathcal{M}$. Then, $Q:=\eta[S_X(\C)]$ and $g:=\eta$ satisfy the requirements by Proposition \ref{equivalences ip minimal}.
\end{proof}

The next remark follows from the fact that for any element $\eta$ in a minimal left ideal $\mathcal{M}\unlhd E(S_X(\C))$ we can find an idempotent $u\in \mathcal{M}$ with $\eta \in u\mathcal{M}$, and then $\Image(u)=\Image(\eta)$ (note that all elements in $u\mathcal{M}$ have the same image).

\begin{remark}\label{Remark Q is the core}
    If the conditions of Proposition \ref{equivalences ip minimal} hold, then the ip-minimal set $Q$ of Proposition \ref{equivalences ip minimal}$(2)$ satisfies $Q\cong \overline{\mathcal{J}}$.
\end{remark}

\begin{lema}
    A subset $Q\subseteq S_X(\C)$ is ip-minimal if and only if every finite $Q_0\subseteq Q$ is ip-minimal. In particular, the union of a chain of ip-minimal subsets of $S_X(\C)$ is ip-minimal.
\end{lema}
\begin{proof}
    It follows by homogeneity of $S_X(\C)$ obtained in Lemma \ref{S_X is homogeneous}.
\end{proof}

%\begin{remark}
    By the previous lemma and Zorn's lemma, there exists some ip-minimal $I_\C\subseteq S_X(\C)$ maximal with respect to inclusion.
%\end{remark}

\begin{lema}\label{morphism are surjective}
    Let $f:I_\C\to K$ be a morphism, where $K$ is an ip-minimal set. Then $f$ is surjective and is therefore an isomorphism. 
\end{lema}
\begin{proof}
    Let $I':=f[I_\C]\subseteq K$. By ip-minimality of $I_\C$, the map $f:I_\C\to I'$ is an isomorphism (in the infinitary definability patterns language). Let $g:I'\to I_\C$ be the inverse of $f$. By Lemma \ref{S_X is homogeneous}, there exists $\overline{g}\in \End(S_X(\C))$ extending $g$. Let $K':=\overline{g}[K]$. Since $K$ is ip-minimal, $\overline{g}\!\!\upharpoonright_K: K\to K'$ is an isomorphism and $K'$ is also ip-minimal. So $K'=I_\C$ by maximality of $I_\C$. Therefore, by injectivity of $\overline{g}\!\!\upharpoonright_K$, $I'=K$. Thus, $f:I_\C\to K$ is onto.
\end{proof}

\begin{cor}\label{unique ip-minimal}
    There is a unique (up to isomorphism) ip-minimal subset $K\subseteq S_X(\C)$ with a morphism $S_X(\C)\to K$. It is the ip-minimal set $I_\C$ described above. Moreover, there is a retraction $S_X(\C)\to I_\C$. 
\end{cor}

\begin{proof}
    Let $g: S_X(\C)\to K$ be a morphism and $K$ an ip-minimal subset (which exists by Corollary \ref{ip-minimal exists}). By Lemma \ref{morphism are surjective}, the map $g\!\!\upharpoonright_{I_\C}:I_\C\to K$ is an isomorphism, which proves uniqueness. 

    For the moreover part, take a morphism $g: S_X(\C)\to I_\C$ (which exists by  the first part). Then $g\!\!\upharpoonright_{I_\C}: I_\C \to I_\C$ is an isomorphism, and so  $f:=(g\!\!\upharpoonright_{I_\C})^{-1}\circ g$ is a retraction from $S_X(\C)$ to $I_\C$. %let $K$ be an ip-minimal subset with a morphism $g: S_X(\C)\to K$ (which exists by Corollary \ref{ip-minimal exists}). Then, the map $f:(g\!\!\upharpoonright_{I_\C})^{-1}\circ g$ is a retraction.
\end{proof}

By Lemma \ref{S_X is homogeneous}, Corollary \ref{unique ip-minimal}, and Remark \ref{Remark Q is the core}, we can assume that $I_\C=\overline{\mathcal{J}}$.

\begin{lema}\label{there is a section}
    Let $\C'\succ \C$. Then the restriction $$r:S_X(\C')\to S_X(\C)$$ is a morphism  of infinitary definability patterns structures and has a section $$s:S_X(\C)\to S_X(\C')$$ which is a morphism of infinitary definability patterns structures.
\end{lema}

\begin{proof}
The fact that $r$ is a morphism is trivial.  To construct $s$, let $\overline{p}=(p_i)_{i<\mu}$ be an enumeration of $S_X(\C)$ and $(a_i)_{i<\mu}$ be a realization of $\overline{p}$ (i.e., $a_i \models p_i$ for $i<\mu$). Let $p:=\tp((a_i)_{i<\mu}/\C)$ and let a type $p':=\tp ((a'_i)_{i<\mu}/\C')$ be a strong heir extension of $p$ (it exists by Fact \ref{fact: existence of heirs}). Then, for any $n<\omega$ and $i_1,\dots i_n<\mu$ we have $c(p_{i_1},\dots,p_{i_n})=c(p'_{i_1},\dots,p'_{i_n})$, where $p_i'=\tp (a'_i/\C')$. Hence, the map $s: S_X(\C)\to S_X(\C')$ given by $s(p_i):=\tp(a_i'/\C')$ is a morphism between the infinitary definability patterns structures, and it is clearly a section of $r$.
\end{proof}

\begin{teor}\label{absoluteness of the core}
    Up to isomorphism in the infinitary definability patterns language, $\overline{\mathcal{J}}$ does not depend on the choice of the $\aleph_0$-saturated, strongly $\aleph_0$-homogeneous model $\C$ for which it is computed.
\end{teor}

\begin{proof}
    It is enough to show that for $\aleph_0$-saturated, strongly $\aleph_0$-homogeneous models  $\C\prec \C'$, $I_\C\cong I_{\C'}$, where $I_{\C'}$ is defined for $\C'$ the same way as $I_\C$ was defined for $\C$. We have the following maps:
    \begin{itemize}
        \item The restriction morphism $r: S_X(\C') \to S_X(\C)$;
        \item A morphism $s:S_X(\C)\to S_X(\C')$ which is a section of $r$, given by Lemma \ref{there is a section};
        \item A retraction $f_\C: S_X(\C)\to I_\C$ given by Corollary \ref{unique ip-minimal};
        \item A retraction $f_{\C'}: S_X(\C')\to I_{\C'}$ given by Corollary \ref{unique ip-minimal}.
    \end{itemize}
    Then, the maps $h_1:=f_\C\circ ( r\!\!\upharpoonright_{I_{\C'}} ): I_{\C'}\to I_\C$ and $h_2:=f_{\C'}\circ ( s\!\!\upharpoonright_{I_\C} ): I_{\C}\to I_{\C'}$ are morphisms. Hence, the maps $h_2\circ h_1:I_{\C'}\to I_{\C'} $ and $h_1 \circ h_2: I_{\C}\to I_{\C}$ are isomorphisms by Lemma \ref{morphism are surjective}.  The first thing implies that $h_1$ is an isomorphism onto its image and the second one that $\Image(h_1)=I_\C$. Therefore, $h_1$ is an isomorphism.
    %This implies that $h_1$ is an isomorphism onto its image. Finally, since $h_1$ and $h_2$ are injective and $h_2[\Image(h_1)]=I_{\C'}=\Image(h_2)$, we have that $h_1$ is surjective. Hence, $h_1$ is an isomorphism.
\end{proof}

One can show that $\overline{\mathcal{J}}$ is precisely Hrushovski's infinitary core (but localized to $X$) considered in \cite[Appendix A]{hrushovski2022definability}; however, we will not use that approach in this paper.

The next corollary was originally proved in \cite{Krupiski2017BoundednessAA} by a much longer argument (also based on contents).

    \begin{cor}\label{absoluteness of the ellis group}
        The Ellis group (considered as a semitopological group with the $\tau$-topology) of the flow $(\aut(\C),S_X(\C))$ does not depend on the choice of $\C$ as long as $\C$ is $\aleph_0$-saturated and strongly $\aleph_0$-homogeneous.
    \end{cor}

    \begin{proof}
    It follows from Propositions \ref{group isomorphism delta}, \ref{delta is homeomorphism}, the definition of ipp-topology, and Theorem \ref{absoluteness of the core}.
\end{proof}

By the above corollary applied for $X:=\C^\omega$ (or, in the multisorted situation, for $X$ being the product of all sorts such that each sort is repeated $\aleph_0$-times), the Ellis group of the flow $S_x(\C)$ does not depend on the choice of $\C$ and is called the {\em Ellis group of the theory $T$}.

\section{Ellis groups of compatible quotients are isomorphic}\label{section: compatible quotients}

%We introduce a natural condition (which we call compatibility) guaranteeing that, for a closed $\aut(\C)$-invariant equivalence relation $F$ on $S_X(\C)$, the Ellis group of the quotient flow $(\aut(\C),S_X(\C)/F)$ is independent of the choice of $\C$ as long as $\C$ is $\aleph_0$-saturated and strongly $\aleph_0$-homogeneous. Hence, in this section we will assume that $\C$ satisfies only those saturation assumptions; and similarly for $\C' \succ \C$.
We introduce a natural condition (which we call compatibility) on closed, invariant equivalence relations $F$ on $S_X(\C)$ and $F'$ on $S_X(\C')$, guaranteeing that the Ellis groups of the quotient flows $(\aut(\C),S_X(\C)/F)$ and $(\aut(\C'),S_X(\C')/F)$ are isomorphic  as long as $\C\prec \C'$ are $\aleph_0$-saturated and strongly $\aleph_0$-homogeneous. Hence, in this section, we will assume that $\C \prec \C'$ satisfy only those saturation assumptions. As usual, $X$ is a $\emptyset$-type-definable set.

\begin{defin}
    Let $F'$ be a closed, $\aut(\C')$-invariant equivalence relation defined on $S_X(\C')$, and $F$ a closed, $\aut(\C)$-invariant equivalence relation defined on $S_X(\C)$. We say that $F'$ and $F$ are {\em compatible} if $r[F']=F$ (i.e., $\{ (r(a),r(b)): (a,b)\in F' \}=F)$, where $r: S_X(\C')\to S_X(\C)$ is the restriction map.
\end{defin}

\begin{teor}\label{compatible implies absolute}
    If $F'$ and $F$ are compatible equivalence relations respectively on $S_X(\C')$ and $S_X(\C)$, then the Ellis group of the flow $(\aut(\C'),S_X(\C')/F')$ is topologically isomorphic to the Ellis group of the flow $(\aut(\C),S_X(\C)/F)$.
\end{teor}
\begin{proof}
   Let $(p_i)_{i<\mu}$ be an enumeration of $S_X(\C)$ and $(a_i)_{i<\mu}$ be a sequence of realizations. Consider the type $\mathfrak{p}:=\tp((a_i)_{i<\mu}/\C)$, and let $\mathfrak{p}':=\tp ((a'_i)_{i<\mu}/\C')$ be a strong heir extension of $\mathfrak{p}$.  For each $i<\mu$ we denote $\tp(a_i'/\C')$ by $p_i'$.

    Let $s\colon S_X(\C)\to S_X(\C')$ be the function given by $s(p_i):=p_i'$. The function $s$ is a section of the restriction map $r$, and since $\mathfrak{p}'$ is a strong heir extension of $\mathfrak{p}$, $s$ is an isomorphism to its image in the infinitary definability  patterns language.

    Choose and idempotent $u \in \mathcal{M}$, where $\mathcal{M}$ is a minimal left ideal of $E(S_X(\C))$. Then $I_{\C}:=\Image(u)$ is ip-minimal by Proposition \ref{equivalences ip minimal}. By Proposition \ref{group isomorphism delta}, there is an isomorphism $\delta$ from $u\mathcal{M}$ to $\aut(I_{\C})$. 
%    The fact that $s \!\! \upharpoonright_{I_{\C}} \colon I_{\C} \to s[I_{\C}]$ is an isomorphism follows from the definition of strong heirs and infinitary definability patterns structures. 
By the previous paragraph,  $s \!\! \upharpoonright_{I_{\C}} \colon I_{\C} \to s[I_{\C}]$ is an isomorphism. 
    On the other hand, by Corollary \ref{unique ip-minimal}, let $I_{\C'}$ be the unique up to isomorphism ip-minimal subset of $S_X(\C')$ for which there is a morphism from $S_X(\C')$ to $I_{\C'}$. By Theorem \ref{absoluteness of the core}, we have $I_{\C'} \cong I_{\C}$, and therefore $I_{\C'} \cong s[I_{\C}]$ (in particular, $s[I_{\C}]$ is ip-minimal in $S_X(\C')$).  
 Thus,  by Lemma \ref{morphism are surjective}, the morphism $\eta:=s \circ u \circ r \colon S_X(\C') \to s[I_{\C}]$ is surjective, and, by Lemma \ref{S_X is homogeneous}, $\eta \in E(S_X(\C'))$. Using Proposition \ref{equivalences ip minimal}, we conclude that $\eta$ is in some minimal left ideal $\mathcal{M'}$ of $E(S_X(\C'))$. Finally, taking an idempotent $u' \in \eta\mathcal{M}'$, we get $\Image(u')=\Image(\eta)= s[I_{\C}]$.
    
    %Let $u\in \mathcal{J}(\mathcal{M})$ be an idempotent where $\mathcal{M}\unlhd E(S_X(\C))$ is a minimal ideal. Then, $I_\C:=Im(u)$ is the infinitary core of $T$ and $u\mathcal{M}\cong \aut(I_\C)$ in the infinitary patterns language(see \cite[Appendix A, Corollary A7]{hrushovski2022definability}). The map $s\!\!\upharpoonright_{I_\C}: I_\C \to s[I_\C]$ is also an isomorphism in the infinitary patterns language and hence $s[I_\C]$ is also the infinitary core. Moreover, we have a surjection $$s\circ u\circ r: S_X(\C') \twoheadrightarrow  s[I_\C] .$$
    %Therefore, there is $\mathcal{M}'\unlhd E(S_X(\C'))$, $u'\in \mathcal{J}(\mathcal{M}')$ such that $Im(u')=s[I_\C]$ and $u'\mathcal{M}'\cong \aut(\C')$.

    The quotient map $\pi\colon S_X(\C')\to S_X(\C')/F' $ induces the following commutative diagram of functions:

      \begin{minipage}[c]{0.9\textwidth}
              \hfill
        \begin{tikzpicture}\label{diagram pi}
        \matrix (m) [matrix of math nodes,row sep=3em,column sep=4em,minimum width=2em]
  {
     \Tilde{\pi}:u'\mathcal{M}' & \Tilde{\pi}(u')\Tilde{\pi}(\mathcal{M}') \\
     \Tilde{\Tilde{\pi}}:\aut(s[I_\C]) & \Tilde{\pi}(u')\Tilde{\pi}(\mathcal{M}')\!\!\upharpoonright_{\pi[s[I_\C]]} \subseteq \Sym(\pi[s[I_\C]]) \\};
  \path[-stealth]
    (m-1-1) edge [->>] node [left] {$\cong$} (m-2-1)
            edge [->>] %node [below] {$\mathcal{B}_t$} 
            (m-1-2)
    (m-2-1.east|-m-2-2) edge [->>] %node [below] {$\mathcal{B}_T$}
            %node [above] {$\exists$} 
            (m-2-2)
    (m-1-2) edge [->>]  node [right] {$\cong$} (m-2-2);
            %edge [dashed,-] (m-2-1);
    \end{tikzpicture}
  \end{minipage}
  \begin{minipage}[c]{0.1\textwidth}
    \hfill (1) 
  \end{minipage}

    where:
    \begin{itemize}
        \item $\aut(s[I_\C])$ is the group of automorphisms in the infinitary definability patterns language.
        \item The map $\Tilde{\pi}$ is given by Fact \ref{epimorphism semigroup to group}.
        \item The isomorphism on the left is given by Proposition \ref{group isomorphism delta} and the fact that $s[I_\C]=\Image(u')$.
        \item The isomorphism on the right is given by Fact \ref{restrictio to image is monomorphism}.
        \item The map $\Tilde{\Tilde{\pi}}\colon\aut(s[I_\C]) \to \Sym(\pi[s[I_\C]])$ is given by $\Tilde{\Tilde{\pi}}(\sigma)(\pi(x)):=\pi(\sigma(x))$ (which is the composition of the inverse of the left side isomorphism, the map $\tilde{\pi}$ and the right side isomorphism).
    \end{itemize}

    Similarly, the quotient map $\rho\colon S_X(\C)\to S_X(\C)/F $ induces the following commutative diagram of functions:
    
      \begin{minipage}[c]{0.8\textwidth}
              \hfill
        \begin{tikzpicture} \label{diagram rho}
        \matrix (m) [matrix of math nodes,row sep=3em,column sep=4em,minimum width=2em]
  {
     \Tilde{\rho}:u\mathcal{M} & \Tilde{\rho}(u)\Tilde{\rho}(\mathcal{M}) \\
     \Tilde{\Tilde{\rho}}:\aut(I_\C) & \Tilde{\rho}(u)\Tilde{\rho}(\mathcal{M})\!\!\upharpoonright_{\rho[I_\C]}\subseteq \Sym(\rho[I_\C]) \\};
  \path[-stealth]
    (m-1-1) edge [->>] node [left] {$\cong$} (m-2-1)
            edge [->>] %node [below] {$\mathcal{B}_t$} 
            (m-1-2)
    (m-2-1.east|-m-2-2) edge [->>] %node [below] {$\mathcal{B}_T$}
            %node [above] {$\exists$} 
            (m-2-2)
    (m-1-2) edge [->>]  node [right] {$\cong$} (m-2-2);
            %edge [dashed,-] (m-2-1);
    \end{tikzpicture}
  \end{minipage}
  \begin{minipage}[c]{0.2\textwidth}
    \hfill (2) 
  \end{minipage}
  
 where:
 \begin{itemize}
     \item  $\aut(I_\C)$ is the group of automorphisms in the infinitary definability patterns language.
     \item The map $\Tilde{\rho}$ is given by Fact \ref{epimorphism semigroup to group}.
    \item The isomorphism on the left is given by Proposition \ref{group isomorphism delta}.
    \item The isomorphism on the right is given by Fact \ref{restrictio to image is monomorphism}.
    \item The map $\Tilde{\Tilde{\rho}}\colon\aut(I_\C) \to \Sym(\rho[I_\C])$ is given by $\Tilde{\Tilde{\rho}}(\sigma)(\rho(x)):=\rho(\sigma(x))$ (which is the composition of the inverse of the left side isomorphism, the map $\tilde{\rho}$ and the right side isomorphism).
 \end{itemize}

 Since the map $r\colon s[I_\C]\to I_\C$ is an isomorphism in the infinitary definability patterns language, it induces an isomorphism $$\overline{r}\colon \aut(s[I_\C])\to \aut(I_\C)$$ given by $\overline{r}(\sigma)(r(p)):=r(\sigma(p))$. Our goal is then to prove that there exists an isomorphism $f$ such that the diagram below commutes:
   \begin{center}
        \begin{tikzpicture}
        \matrix (m) [matrix of math nodes,row sep=3em,column sep=4em,minimum width=2em]
  {
     \aut(s[I_\C]) & \Tilde{\pi}(u')\Tilde{\pi}(\mathcal{M}')\!\!\upharpoonright_{\pi[s[I_\C]]} \\
     \aut(I_\C) & \Tilde{\rho}(u)\Tilde{\rho}(\mathcal{M})\!\!\upharpoonright_{\rho[I_\C]} \\};
  \path[-stealth]
    (m-1-1) edge [->>] node [left] {$\cong$} node [right] {$\overline{r}$} (m-2-1)
            edge [->>] node [above] {$\Tilde{\Tilde{\pi}}$} 
            (m-1-2)
    (m-2-1.east|-m-2-2) edge [->>]  node [above] {$\Tilde{\Tilde{\rho}}$}
            %node [above] {$\exists$} 
            (m-2-2)
    (m-1-2) edge [->>] [dashed] node [right] {$\exists f$ } 
    (m-2-2);
            %edge [dashed,-] (m-2-1);
    \end{tikzpicture}
\end{center}

We first prove that a function $f$ such that the above diagram commutes exists. It is enough to show that $\ker(\Tilde{\Tilde{\pi}})\subseteq \ker(\Tilde{\Tilde{\rho}}\circ \overline{r})$. Note that $\sigma\in \ker(\Tilde{\Tilde{\pi}})$ if and only if for every $p\in s[I_\C]$ we have that $\sigma(p)F' p$. Take an arbitrary $\sigma\in \ker(\Tilde{\Tilde{\pi}})$. By compatibility of $F'$ and F, $\sigma(p)F' p$ implies $ r(\sigma(p))F r(p)$. Hence, for every $p\in s[I_\C]$ we have $\overline{r}(\sigma)(r(p))F r(p)$. Therefore, $\overline{r}(\sigma)$ is in $\ker(\Tilde{\Tilde{\rho}})$ and $\sigma$ is in $\ker(\Tilde{\Tilde{\rho}}\circ \overline{r})$.

To see that $f$ is an isomorphism, it is enough to show that $\ker(\Tilde{\Tilde{\pi}})\supseteq \ker(\Tilde{\Tilde{\rho}}\circ \overline{r})$. So take an arbitrary $\sigma\in \ker(\Tilde{\Tilde{\rho}}\circ \overline{r})$. Then, for every $p\in s[I_\C]$ we have that $r(\sigma(p))F r(p)$. \begin{claim}
    $r(\sigma(p))F r(p)$ implies $\sigma(p)F' p$.
\end{claim}

\begin{proof}[Proof of claim]
    By compatibility of $F'$ and $F$, there are $s_1,s_2\in S_X(\C')$ such that $r(s_1)=r(\sigma(p))$, $r(s_2)=r(p)$ and $s_1F' s_2$. At the same time, since $p,\sigma(p)\in s[I_\C]$, there are $i,j<\mu$ such that: \begin{align*}
       \sigma(p)=p_i' \text{ and } p=p_j';\\
        r(\sigma(p))=p_i \text{ and } r(p)=p_j.        
    \end{align*}
    Hence, $c(s_1,s_2)\supseteq c(r(s_1),r(s_2))=c(r(\sigma(p)),r(p))=c(p_i,p_j)=c(p_i',p_j')=c(\sigma(p),p)$ (where the penultimate equality follows from the fact that $\mathfrak{p}'$ is a strong heir extension of $\mathfrak{p}$). Thus, by Fact \ref{content and ellis semigroup}, there is $\eta\in E(S_X(\C'))$ such that $$\eta(s_1,s_2)=(\sigma(p),p).$$
    Therefore, since $F'$ is $\aut(\C')$-invariant and closed, we conclude that $\sigma(p)F' p$.
\end{proof}
By the claim and the above description of $\ker(\Tilde{\Tilde{\pi}})$, we get that $\sigma\in \ker(\Tilde{\Tilde{\pi}})$.

Moreover, $f$ is a homeomorphism, where $\Tilde{\pi}(u')\Tilde{\pi}(\mathcal{M}')\!\!\upharpoonright_{\pi[s[I_\C]]}$ is equipped with the topology induced from the $\tau$-topology on $\Tilde{\pi}(u')\Tilde{\pi}(\mathcal{M}')$ via the right vertical isomorphism in diagram \hyperref[diagram pi]{(1)}, and  $ \Tilde{\rho}(u)\Tilde{\rho}(\mathcal{M})\!\!\upharpoonright_{\rho[I_\C]}$ with the topology induced from the $\tau$-topology on $ \Tilde{\rho}(u)\Tilde{\rho}(\mathcal{M})$ via the right vertical isomorphism in diagram \hyperref[diagram pi]{(2)}.
To see this, it is enough to show that:
\begin{itemize}
    \item $\tilde{\tilde{\pi}}$ and $\tilde{\tilde{\rho}}$ are topological quotient maps,
%(with the topologies on the right hand sides of the diagram induced by the $\tau$-topologies via the right vertical isomorphisms)
    \item $\bar r$ is a topological isomorphism.
\end{itemize}
The fact that  $\tilde{\tilde{\pi}}$ and $\tilde{\tilde{\rho}}$ are topological quotient maps follows from the fact that in their corresponding diagrams \hyperref[diagram pi]{(1)} and \hyperref[diagram rho]{(2)}, the upper horizontal maps are topological quotient maps by Fact \ref{epimorphism semigroup to group} and the left vertical maps are topological isomorphisms by Proposition \ref{delta is homeomorphism}. The fact that $\bar r$ is a homeomorphism follows trivially by the definition of the ipp-topology and the fact that $\bar r$ is induced by an isomorphism of infinitary definability patterns structures.

We have proved that $f$ is topological isomorphism of groups. Since the right vertical maps in diagrams \hyperref[diagram pi]{(1)} and \hyperref[diagram rho]{(2)} are also topological isomorphisms of groups, we conclude that the Ellis groups $\Tilde{\pi}(u')\Tilde{\pi}(\mathcal{M}')$ and $\Tilde{\rho}(u)\Tilde{\rho}(\mathcal{M})$ are topologically isomorphic, as required.
\end{proof}

\section{Applications to WAP, tame, stable, and NIP context}\label{section: nip stable tame wap}

%In this section, we present several natural equivalence relations that are compatible, allowing us to use Theorem \ref{compatible implies absolute} to obtain absoluteness of their Ellis groups. These are the main results of this paper. 
In this section, we use Theorem \ref{compatible implies absolute} to obtain absoluteness of Ellis groups of several canonical quotients of type-spaces. These are the main results of this paper. 

As in the last section, $\C' \succ \C$, $X$ is a $\emptyset$-type definable set, and $r \colon S_X(\C') \to S_X(\C)$ is the restriction map.
%For the notation used in this section, see Proposition \ref{remark: pist and piNIP} and the comments following it. 

Let $\rho(x,y)$ be a partial type over $\emptyset$ (closed under conjunction) which defines an equivalence relation $E(M)$ on $X(M)$ in a sufficiently saturated (equivalently, in every) model $M$ of $T$.

Recall from Section \ref{subsection: model theory} that $\tilde{E}$ is the equivalence relation on $S_X(\C)$ given by 
$$p \tilde{E} q \iff (\exists a \models p,  b \models q) (\rho(a,b)),$$ 
and $\tilde{E'}$ is the equivalence relation on $S_X(\C')$ given by 
$$p' \tilde{E'} q' \iff (\exists a' \models p',  b' \models q') (\rho(a',b')).$$

Proposition \ref{proposition: uniformly type-definable are compatible} and Corollary \ref{corollary: compatibility for st and NIP} below hold without any saturation assumptions on $\C$ and $\C'$.

\begin{prop}\label{proposition: uniformly type-definable are compatible}
The equivalence relations $\tilde{E'}$ and $\tilde{E}$  are compatible.
\end{prop}

\begin{proof}
The goal is to prove that $r[\tilde{E'}] = \tilde{E}$.

$(\subseteq)$ Consider any $p',q' \in S_X(\C')$ with $p' \tilde{E'}q'$. Then there are $a' \models p'$ and $b' \models q'$ such that $\rho(a',b')$. Hence, $a' \models r(p')$ and $b' \models r(b')$, and we get $r(p') \tilde{E} r(q')$.

$(\supseteq)$ Consider any $p,q \in S_X(\C)$ with $p \tilde{E}q$. The goal is to find some extensions $p',q' \in S_X(\C')$ of $p$ and $q$ respectively, satisfying $p' \tilde{E'} q'$. 

Take $a \models p$ and $b \models q$ such that $\rho(a,b)$. Let $\tp(a'b'/\C')$ be an heir extension of $\tp(ab/\C)$. We claim that $p':=\tp(a'/\C')$ and $q':=\tp(b'/\C)$ do the job. If not, then, by compactness, there are formulas $\varphi(x,y) \in \rho(x,y)$, $\psi_1(x,c') \in p'$, and $\psi_2(x,c') \in q'$ for which there are no $a''$ and $b''$ such that $\psi_1(a'',c') \wedge \psi_2(b'',c') \wedge \varphi(a'',b'')$ (here $\psi_i(x,x')$ is a formula without parameters and $c'$ is a tuple from $\C'$). Then 
$$\psi_1(x,c') \wedge \psi_2(y,c') \wedge \neg (\exists z)(\exists t)(\psi_1(z,c') \wedge \psi_2(t,c') \wedge \varphi(z,t)) \in \tp(a'b'/\C').$$
Since $\tp(a'b'/\C')$ is an heir extension of $\tp(ab/\C)$, there is $c \in \C$ such that 
$$\psi_1(x,c) \wedge \psi_2(y,c) \wedge \neg (\exists z)(\exists t)(\psi_1(z,c) \wedge \psi_2(t,c) \wedge \varphi(z,t)) \in \tp(ab/\C).$$
Taking $z:=a$ and $t:=b$, we get a contradiction with the fact that $\rho(a,b)$.
\end{proof}

Consider the relations $\tilde{E'}^{\textrm{st}}_\emptyset$, $\tilde{E}^{\textrm{st}}_\emptyset$, $\tilde{E'}^{\textrm{NIP}}_\emptyset$, and $\tilde{E}^{\textrm{NIP}}_\emptyset$ defined after Proposition \ref{remark: pist and piNIP}. 

\begin{cor}\label{corollary: compatibility for st and NIP}
\begin{enumerate}
\item The equivalence relations $\tilde{E'}^{\textrm{st}}_\emptyset$ and $\tilde{E}^{\textrm{st}}_\emptyset$  are compatible.
\item The equivalence relations $\tilde{E'}^{\textrm{NIP}}_\emptyset$ and $\tilde{E}^{\textrm{NIP}}_\emptyset$  are compatible.
\end{enumerate}
\end{cor}

\begin{proof}
Both items follow immediately from Proposition \ref{proposition: uniformly type-definable are compatible}.
\end{proof}

From Corollary \ref{corollary: compatibility for st and NIP} and Theorem \ref{compatible implies absolute}, we get the following corollary (note the saturation assumption).

\begin{cor}
\begin{enumerate}
\item     The Ellis group of $S_X(\C)/\tilde{E}^{\textrm{st}}_\emptyset$ (treated as a semitopological group with the $\tau$-topology) does not depend on the choice of $\C$ as long as $\C$ is at least $\aleph_0$-saturated and strongly $\aleph_0$-homogeneous.
\item	The Ellis group of $S_X(\C)/\tilde{E}^{\textrm{NIP}}_\emptyset$ (treated as a semitopological group with the $\tau$-topology) does not depend on the choice of $\C$ as long as $\C$ is at least $\aleph_0$-saturated and strongly $\aleph_0$-homogeneous.
\end{enumerate}
\end{cor}

%%%%%%%%%%%%%%%%%%%%%%%%%%%%%%%%%%%%%%%%%%%%%%%%%
%%%%%%%%%%%%%%%%%%%%%%%%%%%%%%%%%%%%%%%%%%%%%%%%%
\begin{comment}
By the same argument, we get the following:

\begin{prop}
The equivalence relations $\tilde{E'}^{\textrm{NIP}}_\emptyset$ and $\tilde{E}^{\textrm{NIP}}_\emptyset$  are compatible.
\end{prop}

\begin{cor}
    The Ellis group of $S_X(\C)/\tilde{E}^{\textrm{NIP}}_\emptyset$ (treated as a semitopological group with the $\tau$-topology) does not depend on the choice of $\C$ as long as $\C$ is at least $\aleph_0$-saturated and strongly $\aleph_0$-homogeneous.
\end{cor}

\end{comment}
%%%%%%%%%%%%%%%%%%%%%%%%%%%%%%%%%%%%%%%%%%%%%%%
%%%%%%%%%%%%%%%%%%%%%%%%%%%%%%%%%%%%%%%%%%%%%%%%

Next, we will show that the same is true for the equivalence relations $\fwapp$ and $\fwap$ described at the very end of Section \ref{section: topological dynamics}. However, this time we will need $\C$ and $\C'$ to be at least $(\aleph_0+\lambda)^+$-saturated and strongly $(\aleph_0+\lambda)^+$-homogeneous (where lambda is such that $X\subseteq \C^\lambda$). In particular, if $\lambda$ is finite, then $\aleph_1$-saturation and strong $\aleph_1$-homogeneity is enough. 

\begin{remark}\label{remark: reduction to big C'}
    For the proof of the next theorem, without loss of generality we will assume that $\C$ is $\C'$-small in the sense that the degree of saturation of $\C'$ is bigger than $|\C|$. This is because if the theorem holds under this assumption, we can take a monster model $\C'' \succ \C'$ in which both $\C$ and $\C'$ are small, and apply the result to the pairs $\C'' \succ \C'$ and $\C'' \succ \C$. Namely, for $r_1\colon S_X(\C') \to S_X(\C)$, $r_2 \colon S_X(\C'') \to S_X(\C')$, and $r_3 \colon S_X(\C'') \to S_X(\C)$ being the restriction maps, the theorem yields $r_2[F_{\textrm{WAP}}'']=\fwapp$ and $r_3[F_{\textrm{WAP}}'']=\fwap$. As $r_1[r_2[F_{\textrm{WAP}}'']]=r_3[F_{\textrm{WAP}}'']$, we conclude that $r_1[\fwapp] =\fwap$.
\end{remark}

\begin{teor}\label{fwapp and fwap}
The equivalence relations $\fwapp$ and $\fwap$ are compatible as long as $\C$ and $\C'$ are at least $(\aleph_0+\lambda)^+$-saturated and strongly $(\aleph_0+\lambda)^+$-homogeneous.
\end{teor}

\begin{proof}
    We first prove that $\fwap\subseteq r[\fwapp]$. It suffices to show that $ r[\fwapp]$ is a closed, $\aut(\C)$-invariant equivalence relation on $S_X(\C)$ with WAP quotient.
    \begin{itemize}
        \item Closedness is clear by continuity of $r$ and compactness of type spaces.
        \item Equivalence relation: Take $a=r(\alpha)$, $b=r(\beta)=r(\beta')$ and $c=r(\gamma)$ where $\alpha \fwapp \beta$ and $\beta' \fwapp \gamma$. Let $(p_i)_{i<\mu}$ be an enumeration of $S_X(\C)$ and $(a_i)_{i<\mu}$ be a sequence of realizations. Consider the type $\mathfrak{p}:=\tp((a_i)_{i<\mu}/\C)$, and let $\mathfrak{p}':=\tp ((a'_i)_{i<\mu}/\C')$ be a strong heir extension of $\mathfrak{p}$. For each $i<\mu$ we denote $\tp(a_i'/\C')$ by $p_i'$. Obviously, $a=p_{i_1}$, $b=p_{i_2}$, and $c=p_{i_3}$ for some $i_1,i_2,i_3<\mu$.
        \begin{claim}
            $p'_{i_1}\fwapp p'_{i_2} \fwapp p'_{i_3}$.
        \end{claim}
        \begin{proof}[Proof of claim]
            Clearly, $r(p'_i)=p_i$ for all $i<\mu$. Since $\mathfrak{p}'$ is a strong heir extension of $\mathfrak{p}$, we have 
            \begin{align*}
                c(p'_{i_1},p'_{i_2})&=c(p_{i_1},p_{i_2}) = c(r(\alpha),r(\beta)) \subseteq c(\alpha,\beta),\\
                c(p'_{i_2},p'_{i_3})&=c(p_{i_2},p_{i_3}) = c(r(\beta'),r(\gamma))\subseteq c(\beta',\gamma).
            \end{align*}
            Thus, by Fact \ref{content and ellis semigroup}, \begin{align*}
                \exists \eta_1 \in E(S_X(\C'))&[ \eta_1 (\alpha, \beta)=(p'_{i_1},p'_{i_2})],\\
                \exists \eta_2 \in E(S_X(\C'))&[ \eta_2 (\beta', \gamma)=(p'_{i_2},p'_{i_3})].
            \end{align*}
            Since the relation $\fwapp$ is $\aut(\C')$-invariant and closed, $ \alpha \fwapp \beta$, and  $\beta'\fwapp \gamma $, we conclude that $p'_{i_1}\fwapp p'_{i_2} \fwapp p'_{i_3}$.
        \end{proof}
        By the claim, $p'_{i_1}\fwapp p'_{i_3}$, so $a=r(p'_{i_1}) r[\fwapp] r(p'_{i_3})=c$.
        \item $\aut(\C)$-invariant: Take an arbitrary $\sigma\in \aut(\C)$ and extend it to $\sigma'\in \aut(\C')$. Consider any $a,b\in S_X(\C)$ such that $a r[\fwapp] b$ and let $\alpha$, $\beta \in S_X(\C')$ be such that $r(\alpha)=a$, $r(\beta)=b$, and $\alpha \fwapp \beta$. Then, $$ \sigma(a)=\sigma(r(\alpha))=r(\sigma'(\alpha))r[\fwapp] r(\sigma'(\beta))=\sigma(r(\beta))=\sigma(b). $$
        \item $S_X(\C)/r[\fwapp]$ is WAP: Assume for a contradiction that there is a function $f \in C ( S_X(\C)/r[\fwapp] )$ which is not WAP. That is, by Fact \ref{double limit}, there is a net $(\sigma_i)_{i\in \I}\subseteq \aut( \C)$ such that the functions $\sigma_if$ converge pointwise to some function $g\notin C ( S_X(\C)/r[\fwapp] )$. Note that $r^{-1}[r[\fwapp]]\supseteq \fwapp$ is a closed $\aut(\C'/\{ \C \})$-invariant equivalence relation on $S_X(\C')$, and $r$ induces a homeomorphism $$\Tilde{r}: \bslant{S_X(\C')}{r^{-1}[r[\fwapp]]}\to \bslant{S_X(\C)}{r[\fwapp]} $$ satisfying $$ \Tilde{r}\left(\sigma'\left(\bslant{p}{r^{-1}[r[\fwapp]]}\right) \right)=\sigma'\!\!\upharpoonright_\C\left(\bslant{r(p)}{r[\fwapp]}\right) $$ for all $\sigma'\in\aut(\C'/\{ \C\})$. 
%In particular, we have a homomorphism
%        $$ \left( \aut(\C'/\{ \C \}), \bslant{ S_X(\C')} {r^{-1}[r[\fwapp]]} \right) \to \left(\aut(\C), \bslant {S_X(\C)}{r[\fwapp]} \right) $$ given by 
%$$ \left(\sigma', \bslant{p}{r^{-1}[r[\fwapp]]} \right)\mapsto\left(\sigma'\!\!\upharpoonright_\C, \bslant{r(p)}{r[\fwapp]} \right). $$

        For every $i\in \I$ choose an extension $\sigma'_i\in \aut(\C'/\{ \C \})$ of $\sigma_i$. The above homeomorphism $\Tilde{r}$ together with $f$ induce a function $f'\in C ( S( \C')/ r^{-1}[r[\fwapp]] )$ given by $$f'\left( \bslant{p}{r^{-1}[r[\fwapp]]} \right):=f\left( \bslant{r(p)}{r[\fwapp]} \right). $$ By construction, the net $(\sigma'_if')_{i\in I}$ converges pointwise to a function $$g'\notin C ( S( \C')/ r^{-1}[r[\fwapp]] ).$$ Hence, the flow $$ \left( \aut(\C'/\{ \C \}), \bslant{ S_X(\C')} {r^{-1}[r[\fwapp]]} \right)$$  is not WAP, which implies that $(\aut(\C'), S_X(\C')/\fwapp) $ is not WAP (because WAP is preserved under decreasing the acting group and under taking quotients of flows, which follows by  Remark \ref{remark: transfer of WAP and tame}), a contradiction.
    \end{itemize}

    Now, we prove $\fwap\supseteq r[\fwapp]$. This is equivalent to $r^{-1}[\fwap]\supseteq \fwapp$. Note that $r^{-1}[\fwap]$ is a closed equivalence relation on $S_X(\C')$ but it might not be $\aut(\C')$-invariant. To solve it, we consider the equivalence relation $$ F:= \bigcap_{\sigma \in \aut(\C')} \sigma(r^{-1}[\fwap] ). $$ Then, it is enough to show that $F\supseteq \fwapp$, which is equivalent to the flow $$(\aut(\C'),  S_X(\C')/F)$$ being WAP.

    Assume for a contradiction that there is $f\in C(S_X(\C')/F)$ which is not WAP. Let $\overline{f}:S_X(\C')\to \R$ be given by $\overline{f}=f\circ \pi_F$, where $\pi_F:S_X(\C')\to S_X(\C')/F$ is the quotient map. Then, $\overline{f}\in C( S_X(\C'))$ and it is not a WAP function by  Remark \ref{remark: transfer of WAP and tame}. By Fact \ref{double limit}, there are $(\sigma_n)_{n<\omega}\subseteq \aut(\C')$ and $(c'_m )_{m<\omega}\subset X(\C')$ such that \begin{equation}\label{dlp f}
        \lim_n\lim_m (\sigma_n\overline{f}) (\tp(c'_m/\C')) \neq \lim_m\lim_n (\sigma_n\overline{f}) (\tp(c'_m/\C'))
    \end{equation} where both limits exist. Note that here we are using that the types over $\C'$ of the elements of $X(\C')$ form a dense subset of $S_X(\C')$, which uses that $X$ lives on $\C'$-small 
%(even $\C$-small) 
tuples.

%%%Krzys: In the next sentence, only $(\aleph_0+\lambda)$-saturation is needed.
    Consider the set $N:=\{ c'_m: m<\omega \}\cup \{ \sigma_n^{-1}(c'_m): m,n<\omega \}.$ By $(\aleph_0+\lambda)$-saturation of $\C$, we may assume that $N\subset \C$ (just find $\sigma\in\aut(\C')$ such that $\sigma[N]\subset \C$ and replace $N$ by $\sigma[N]$, $c'_m$ by $\sigma(c'_m)$, and $\sigma_n$ by $\sigma\circ \sigma_n$). Note that it might happen that none of the restrictions $\sigma_n\!\!\upharpoonright_\C$ is an automorphism of $\C$. However, for every $n<\omega$, by strong $(\aleph_0+\lambda)^+$-homogeneity of $\C$ (and strong $\lvert \C\rvert^+$-homogeneity of $\C')$, we can replace $\sigma_n$ by $\sigma'_n\in \aut(\C')$ so that $\sigma_n'^{-1}$ extends  $\sigma_n^{-1}\!\!\upharpoonright_{\{c'_m: m<\omega \}}$ and the restriction of $\sigma'_n$ to $\C$ is in $\aut(\C)$, so we may assume that, for every $n<\omega$, the restriction $\sigma_n\!\!\upharpoonright_\C=:\tau_n$ is an element of $\aut(\C)$.

    Let $\mathcal{H}:= \{ p\in S_X(\C): p \text{ is invariant over } N \}$; we enumerate $\mathcal{H}$ as $(p_i)_{i<\mu}$ and choose $a_i\models p_i$ for every $i<\mu$. Consider the type $\mathfrak{p}:=\tp((a_i)_{i<\mu}/\C)$ and let $\mathfrak{p}':=\tp ((a'_i)_{i<\mu}/\C')$ be a strong heir extension of $\mathfrak{p}$ in the language $\mL_N$ (i.e., $\mathcal{L}$ expanded by constants from $N$).
%%%Krzys: In comment below, we need $(\aleph_0+\lambda)^+$-saturation, not just $\aleph_1$-saturation, as $N$ is of size $\aleph_0 + \lambda$. 
($\mathfrak{p}'$ exists by $\aleph_0$-saturation of $\C$ in the language $\mL_N$ which follows from $(\aleph_0+\lambda)^+$-saturation of $\C$ in the language $\mL$.) We denote $\tp(a_i'/\C')$ by $p_i'$ for each $i<\mu$. Since $\tp(a_i/\C)$ is $N$-invariant, $p_i'$ is the unique $N$-invariant extension of $p_i$ to $\C'$. Hence, the set $\mathcal{H}':=\{p_i'\}_{i<\mu}$ is precisely the set of all types in $S_X(\C')$ invariant over $N$, so it is closed in $S_X(\C')$. 

    Now, we define $h:\mathcal{H} \to \R$ by $h(p_i):=\overline{f}(p_i')$. The function $h$ belongs to $C(\mathcal{H})$ since for each closed interval $I\subseteq \R$ we have that $h^{-1}[I]=r[\overline{f}^{-1}[I] \cap \mathcal{H}']$ is a closed subset of $S_X(\C)$. Note that for each $c'\in N$ and $i<\mu$ such that $p_i=\tp(c'/\C)$ we have $p_i'=\tp(c'/\C')$, and so
    \begin{equation}\label{6.2}
    \begin{split}
        (\sigma_n\overline{f})( \tp( c'_m/\C') ) = \overline{f}(\tp(\sigma_n^{-1}(c'_m)/\C'))=\\=h(\tp(\sigma_n^{-1}(c'_m)/\C))= (\tau_nh)( \tp( c'_m/\C)). 
    \end{split}
    \end{equation} Using (\ref{dlp f}) and (\ref{6.2}), we have
    \begin{equation}\label{6.3}
        \lim_n\lim_m (\tau_n h) (\tp(c'_m/\C)) \neq \lim_m\lim_n (\tau_n h) (\tp(c'_m/\C))
    \end{equation} where both limits exist.

    \begin{claim}
        For $p_i,p_j \in\mathcal{H}$, if $p_i \fwap p_j$, then $h(p_i)=h(p_j)$.
    \end{claim}
    \begin{proof}[Proof of claim]
        We show that $p_i' F p_j'$. Choose an arbitrary $\sigma\in\aut(\C')$. We have that $$ c(p_i,p_j)=c(p_i',p_j')=c(\sigma(p_i'),\sigma(p_j'))\supseteq c(r(\sigma(p_i')),r(\sigma(p_j'))) .$$ Hence, there is $\eta\in E(S_X(\C))$ such that $\eta(p_i,p_j)=(r(\sigma(p_i')),r(\sigma(p_j')))$, which implies that $ r(\sigma(p_i')) \fwap r(\sigma(p_j')) $ (because $\fwap$ is $\aut(\C)$-invariant and closed). We then have $\sigma(p_i') r^{-1}[\fwap]\sigma(p_j')$, and since $\sigma$ was arbitrary, we conclude that $p_i' F p_j'$. Therefore, since $\overline{f}=f\circ\pi_F$, we obtain that $\overline{f}(p'_i)=\overline{f}(p'_j)$, so $h(p_i)=h(p_j)$.
    \end{proof}

    Clearly, $\mathcal{H}/\fwap$ is a closed subset of $S_X(\C)/\fwap$ and, by the claim, $h=g\circ \rho$ for some $g\in C(\mathcal{H}/\fwap )$ where $\rho: \mathcal{H} \to \mathcal{H}/\fwap$ the quotient map. Tietze's extension theorem yields a function $\overline{g}\in C(S_X(\C)/\fwap)$ extending $g$. By construction (in particular, using (\ref{6.3})), we get\begin{equation}\label{6.4}
        \lim_n\lim_m (\tau_n\overline{g}) \left(\bslant{\tp(c'_m/\C)}{\fwap}\right) \neq \lim_m\lim_n (\tau_n\overline{g}) \left(\bslant{\tp(c'_m/\C)}{\fwap}\right)
    \end{equation} where both limits exist, which by Fact \ref{double limit} 
%and the fact that $\sigma_n \!\! \upharpoonright_{\C} \, \in \aut(\C)$ 
implies that the flow $(\aut(\C),S_X(\C)/\fwap)$ is not WAP, a contradiction.
\end{proof}

From the previous theorem and Theorem \ref{compatible implies absolute}, the following corollary follows immediately.

\begin{cor}
    The Ellis group of $S_X(\C)/\fwap$ (treated as a semitopological group with the $\tau$-topology) does not depend on the choice of $\C$ as long as $\C$ is at least $(\aleph_0+\lambda)^+$-saturated and strongly $(\aleph_0+\lambda)^+$-homogeneous.
\end{cor}

Similarly, and under the same saturation assumptions, for the equivalence relations $\ftaa$ and $\fta$ we have the following:

\begin{teor}\label{ftaa and fta}
The equivalence relations $\ftaa$ and $\fta$ are compatible as long as $\C$ and $\C'$ are at least $(\aleph_0+\lambda)^+$-saturated and strongly $(\aleph_0+\lambda)^+$-homogeneous.
\end{teor}

\begin{proof}
By an obvious analog of Remark \ref{remark: reduction to big C'}, without loos of generality we can assume that $\C$ is $\C'$-small.

    We first prove that $\fta\subseteq r[\ftaa]$. It suffices to show that $ r[\ftaa]$ is a closed, $\aut(\C)$-invariant equivalence relation on $S_X(\C)$ with tame quotient. The fact that $r[\ftaa]$ is a closed, $\aut(\C)$-invariant equivalence relation on $S_X(\C)$ follows by the same arguments as in the WAP context (see the proof Theorem \ref{fwapp and fwap}).
    
    To show that $S_X(\C)/r[\ftaa]$ is tame, suppose for a contradiction that there is a function $f \in C ( S_X(\C)/r[\ftaa] )$ which is not tame. That is, there is a sequence $(\sigma_i)_{i<\omega}\subset\aut(\C)$ such that $(\sigma_i f)_{i<\omega}$ is an independent sequence. 
    Then we apply the corresponding part of the proof of Theorem \ref{fwapp and fwap}, replacing ``WAP" by ``tame" and noticing that by construction the sequence $(\sigma'_if')_{i<\omega}$ is independent, which leads to a contradiction.

    Now, we prove $\fta\supseteq r[\ftaa]$. This is equivalent to $r^{-1}[\fta]\supseteq \ftaa$. We define a closed, $\aut(\C')$-invariant equivalence relation $$ F:= \bigcap_{\sigma \in \aut(\C')} \sigma(r^{-1}[\fta] ). $$ Then, it is enough to show that $F\supseteq \ftaa$, which is equivalent to the flow $$(\aut(\C'),  S_X(\C')/F)$$ being tame.

    Assume for a contradiction that there is $f\in C(S_X(\C')/F)$ which is not tame. Let $\overline{f}:S_X(\C')\to \R$ be given by $\overline{f}=f\circ \pi_F$. Then, $\overline{f}\in C( S_X(\C'))$ and it is not a tame function  by  Remark \ref{remark: transfer of WAP and tame}. That is, there exist $r<s\in \R$, $(\sigma_n)_{n<\omega}\subset\aut(\C')$ and $\{c'_{P,M}: P,M\subset_{\text{fin}} \omega \text{ disjoint} \}\subset X(\C')$ such that for any finite disjoint $P,M\subset \omega$ \begin{equation}\label{indep f}
%        \models \{ \sigma_i\overline{f}(\tp(c'_{P,M}/\C'))<r: i\in P\} \cup \{\sigma_i\overline{f}(\tp(c'_{P,M}/\C'))>s: i\in M \}.
 (\sigma_n\overline{f})(\tp(c'_{P,M}/\C'))<r \text{ if } n\in P \; \text{ and } \; (\sigma_n\overline{f})(\tp(c'_{P,M}/\C'))>s \text{ if } n\in M.
    \end{equation}
    The fact that we can choose $c'_{P,M}\in X(\C')$ follows from the fact that the types over $\C'$ of elements of $X(\C')$ form a dense subset of $S_X(\C')$ and the second part of Defnition \ref{defin: indep functons}.
    %By Grothendieck, there are $(\sigma_n)_{n<\omega}\subseteq \aut(\C)$ and $(c'_m )_{m<\omega}\subset \C'$ such that $$ \lim_n\lim_m \sigma_n\overline{f} (\tp(c'_m/\C')) \neq \lim_m\lim_n \sigma_n\overline{f} (\tp(c'_m/\C')) $$ where both limits exist.

    Consider the set $$N:=\{c'_{P,M}: P,M\subset_{\text{fin}} \omega \text{ disjoint} \}\cup \{\sigma_n^{-1}(c'_{P,M}): i<\omega,\;  P,M\subset_{\text{fin}} \omega \text{ disjoint} \}.$$ 
%%%Krzys: Again  $(\aleph_0+\lambda)$-sat. is enough in the next sentence.
%By $(\aleph_0+\lambda)$-saturation of $\C$, applying an automorphism of $\C'$, we may assume that $N\subset \C$. Note that it might happen that none of the restrictions $\sigma_n\!\!\upharpoonright_\C$ is an automorphism of $\C$. However, for every $n<\omega$, by strong $(\aleph_0+\lambda)^+$-homogeneity of $\C$ (and strong $\lvert \C \rvert^{+}$-homogeneity of $\C'$), we can replace $\sigma_n$ by $\sigma'_n\in \aut(\C')$ so that $\sigma_n'^{-1}$ extends $\sigma_n^{-1}\!\!\upharpoonright_{\{c'_{P,M}: P,M\subset_{\text{fin}} \omega \text{ disjoint}\}}$ and the restriction of $\sigma'_n$ to $\C$ is in $\aut(\C)$, so we may assume that, for every $n<\omega$, $\sigma_n\!\!\upharpoonright_\C$ is an element of $\aut(\C)$.
As in the proof of Theorem \ref{fwapp and fwap}, using $(\aleph_0+\lambda)$-saturation and strong $(\aleph_0+\lambda)^+$-homogeneity of $\C$ (together with strong $\lvert \C \rvert^{+}$-homogeneity of $\C'$), we may assume that $N\subset \C$ and, for every $n<\omega$, $\tau_n:=\sigma_n\!\!\upharpoonright_\C$ is an element of $\aut(\C)$.

%    Let $\mathcal{H}:= \{ p\in S_X(\C): p \text{ is invariant over } N \}$, we enumerate $\mathcal{H}=( p_i)_{i<\mu}$ and choose $a_i\models p_i$ for every $i<\mu$.  Consider the type $\mathfrak{p}:=\tp((a_i)_{i<\mu}/\C)$ and let $\mathfrak{p}':=\tp ((a'_i)_{i<\mu}/\C')$ be a strong heir extension of $\mathfrak{p}$ in the language $\mL_N$. We denote $\tp(a_i'/\C')$ by $p_i'$ for each $i<\mu$. Since $\tp(a_i/\C)$ is $N$-invariant, $p_i'$ is the unique $N$-invariant extension of $p_i$ to $\C'$. Hence, the set $\mathcal{H}':=(p_i')_{i<\mu}$ is closed in $S_X(\C')$. 

    Then we apply the corresponding part of the proof of Theorem \ref{fwapp and fwap}, where the formulas (\ref{6.2}), (\ref{6.3}), and (\ref{6.4}) are replaced by:
\begin{equation} 
\begin{split}
(\sigma_n\overline{f})( \tp( c'_{P,M}/\C') ) = \overline{f}(\tp(\sigma_n^{-1}(c'_{P,M})/\C'))=\\ h(\tp(\sigma_n^{-1}(c'_{P,M})/\C))= (\tau_nh)( \tp( c'_{P,M}/\C) ),
\end{split}
\end{equation}
for any finite disjoint $P,M\subset \omega$ 
\begin{equation}
(\tau_nh)(\tp(c'_{P,M}/\C))<r \text{ if } n\in P \; \text{ and } \; (\tau_nh)(\tp(c'_{P,M}/\C))>s \text{ if } n\in M,
\end{equation} 
\begin{equation}
\begin{split}
(\tau_n\overline{g}) \left(\bslant{\tp(c'_{P,M}/\C)}{F_{\text{Tame}}}\right)<r \text{ if } n\in P, \; \text{ and } \\  (\tau_n\overline{g})\left(\bslant{\tp(c'_{P,M}/\C)}{F_{\text{Tame}}}\right)>s \text{ if } i\in M,
\end{split}
\end{equation} 
respectively. The last property 
%together with the fact that $\sigma_n \!\! \upharpoonright_{\C} \, \in \aut(\C)$ 
implies that the flow $(\aut(\C),S_X(\C)/\fta)$ is not tame, a contradiction.
\end{proof}

Again, from the previous theorem and Theorem \ref{compatible implies absolute}, the following corollary follows immediately.

\begin{cor}
    The Ellis group of $S_X(\C)/\fta$ (treated as a semitopological group with the $\tau$-topology) does not depend on the choice of $\C$ as long as $\C$ is at least $(\aleph_0+\lambda)^+$-saturated and strongly $(\aleph_0+\lambda)^+$-homogeneous.
\end{cor}

\section{Stable vs WAP, and NIP vs tame}\label{section: Stable vs WAP}

In this section, we will study the relationship between the equivalence relations $\fta$ and $\fwap$ and the finest $\emptyset$-type-definable equivalence relations on $X$ with NIP and stable quotients, respectively. Assume that $\C$ is $(\aleph_0+\lambda)$-saturated and strongly $(\aleph_0+\lambda)$-homogeneous (where $X\subseteq \C^\lambda)$.

Let $E$ be a $\emptyset$-type-definable equivalence relation on $X$. In \cite{10.1215/00294527-2022-0023}, we defined $\mathcal{F}_{X/E}$ as the family of all functions $f\colon X \times \mathfrak{C}^m \to \mathbb{R}$ which factor through $X/E\times\mathfrak{C}^m$ and can be extended to a continuous logic formula $\C^{\lambda} \times \C^m \to \mathbb{R}$ over $\emptyset$, where $m$ ranges over $\omega$. By \cite[Corollary 2.2]{10.1215/00294527-2022-0023}, we know that the quotient $X/E$ is stable if and only if every $f\in \mathcal{F}_{X/E}$ is stable. 

For $f(x,y)\in \mathcal{F}_{X/E}$ and $b\in \C^{\lvert y \rvert}$, by $f_b$ we denote the function $f_b\colon S_{X/E}(\C)\to \R$ given by $f_b(p):=f(a,b)$ for any $a/E'\models p$ (where $a \in \C' \succ \C$, $E':=E(\C')$, and $f$, being a restriction of a continuous logic formula, is treated as a function $X(\C') \times \C'^{|y|} \to \mathbb{R}$). Since $f$ is a restriction of a continuous logic formula, $f_b$ is continuous.

As mentioned in the introduction, the connection between stability and WAP is well-known (discovered in \cite{yaacov2014model, ben2016weakly}). This connection is still present for the hyperdefinable set $X/E$.

\begin{prop}\label{wap iff stable formula}
    Let $f(x,y)\in \mathcal{F}_{X/E}$, and let $b\in \C^{\lvert y \rvert}$. 
%We denote by $f_b$ the function $f_b\colon S_{X/E}(\C)\to \R$ given by $f_b(p):=f(a,b)$ for any $a/E\models p$. 
Then the following are equivalent:
    \begin{enumerate}
        \item $f(x,y)$ is stable.
        \item For all $b\in \C^{\lvert y \rvert}$ the function $f_b$ is WAP (for the flow $(\aut(\C),S_{X/E}(\C))$).
    \end{enumerate}
\end{prop}
\begin{proof}
    $(1)\Rightarrow (2)$ If the function $f_b$ is not WAP, by Fact \ref{double limit}, there is a sequence $(a_n)_{n<\omega}\subset X$ and a sequence of automorphisms $(\sigma_m)_{m<\omega}\subset \aut(\C)$ such that $$\lim_m\lim_n f(a_n,\sigma_m(b))\neq \lim_n\lim_m f(a_n,\sigma_m(b)) .$$ Note that we are using that the realized types are dense in $S_X(\C)$.
 Assume that $\lim_m\lim_n f(a_n,\sigma_m(b))>\lim_n\lim_m f(a_n,\sigma_m(b))$. (The opposite case is analogous.)
Then for some real numbers $r<s$ we have $$\lim_m\lim_n f(a_n,\sigma_m(b))>s \text{ and } \lim_n\lim_m f(a_n,\sigma_m(b))<r.$$
 
Let us denote $\sigma_m(b)$ by $b_m$. It is clear that we can choose a subsequence $(a_i',b_i')_{i<\omega}$ from $(a_n,b_n)_{n<\omega}$ such that $f(a'_i,b'_j)>s$ whenever $i>j$ and $f(a'_i,b'_j)<r$ whenever $i<j$. That is, the sequence $(a_i',b_i')_{i<\omega}$ witnesses unstability of $f(x,y)$. 

    $(2)\Rightarrow (1)$ If $f(x,y)$ is unstable, we can find an indiscernible sequence $(a_i,b_i)_{i<\omega}$ with $a_i\in X$ and $b_i\in \C^{\lvert y \rvert}$ such that $f(a_i,b_j)\neq f(a_j,b_i)$ for some/all $i<j$. By indiscernibility, for each $i<\omega$ there is $\sigma_i \in \aut(\C)$ such that $\sigma_i(b_0)=b_i$ and there exist real numbers $r<s$ such that: ($f(a_i,b_j)=r<s=f(a_j,b_i)$ for all $i<j$) or ($f(a_i,b_j)=s>r=f(a_j,b_i)$ for all $i<j$). Hence, $f_{b_0}$ is not a WAP function by Fact \ref{double limit}.
\end{proof}

\begin{cor}\label{cor: stable iff wap}
    The flow $(\aut(\C), S_{X/E}(\C))$ is WAP if and only if $X/E$ is stable.
\end{cor}
\begin{proof}
    $(\Rightarrow)$ By assumption, for any $f(x,y)\in\mathcal{F}_{X/E}$ and $b\in\C^{\lvert y\rvert}$, $f_b$ is WAP, so $f(x,y)$ is stable by Proposition \ref{wap iff stable formula}. Hence, $X/E$ is stable by \cite[Corollary 2.2]{10.1215/00294527-2022-0023}.
    
    $(\Leftarrow)$ By assumption and \cite[Corollary 2.2]{10.1215/00294527-2022-0023}, every function $f(x,y)\in\mathcal{F}_{X/E}$ is stable, so for any $b\in \C^{\lvert y \rvert}$ the function $f_b$ is WAP by Proposition \ref{wap iff stable formula}. Thus, since by \cite[Proposition 2.1]{10.1215/00294527-2022-0023} the family of functions $\{ f_b: f\in \mathcal{F}_{X/E},b\in \C^{\lvert y \rvert}  \}$ separates points in $S_{X/E}(\C)$, we conclude that $(\aut(\C), S_{X/E}(\C))$ is WAP by Fact \ref{separating points: WAP}.
\end{proof}

%Note that given a $\emptyset$-type-definable equivalence relation $E$ on $\C$  we can define a $\emptyset$-type-definable equivalence relation $\Tilde{E}$ on $S_X(\C)$ given by $$ p\Tilde{E}q\iff \exists a\models p, b\models q \left( \tp(\bslant{a/E}{\C})=\tp(\bslant{b/E}{\C}) \right)$$ and $\Tilde{E}$ satisfies that $S_{\C/E}(\C)\cong S_\C(\C)/\Tilde{E}$.

Similarly, by \cite[Lemma 5.7]{fernández2024ndependent}, we know that the quotient $X/E$ has NIP if and only if every $f\in \mathcal{F}_{X/E}$ has NIP (see \cite[Definition 5.4]{fernández2024ndependent} for $n=1$). As mentioned in the introduction, the connection between NIP and tameness flows is well-known (first noticed independently in \cite{definably_Amenable_nip}, \cite{ibarlucia2016dynamical}, and \cite{Khanaki2020-KHASTN}). This connection is still present for the hyperdefinable set $X/E$.

\begin{prop}\label{tame iff nip formula}
    Let $f(x,y)\in \mathcal{F}_{X/E}$, and let $b\in \C^{\lvert y \rvert}$. 
%We denote by $f_b$ the function $f_b:S_{X/E}(\C)\to \R$ given by $f_b(p)=f(a,b)$ for any $a/E\models p$. 
Then the following are equivalent:
    \begin{enumerate}
        \item $f(x,y)$ has NIP.
        \item For all $b\in \C^{\lvert y \rvert}$ the function $f_b$ is tame.
    \end{enumerate}
\end{prop}

\begin{proof}
    $(1)\Rightarrow (2)$ If $f_b$ is not tame for some $b\in \C^{\lvert y \rvert}$, then there is a sequence $(\sigma_i)_{i<\omega}\subset \aut(\C)$ such that the sequence of functions $(f(x,\sigma_i(b)))_{i<\omega}$ is independent on $X$, so $f$ has IP by compactness. %we construct a witness for $IP$ using the independent family $(f(x,\sigma_i(b)))_{i<\omega}$ and compactness.

    $(2)\Rightarrow (1)$ 
%If $f(x,y)$ has $IP$, then we can find an element $a\in X$, an indiscernible sequence $(b_i)_{i<\omega}\subset \C^{\lvert y \rvert}$ and $r<s\in \mathbb{R}$ such that $f(a,b_i)<r$ if and only if $i$ is even, and $f(a,b_i)>s$ if and only if $i$ is odd 
%(this can be easily deduced from \cite[Proposition 4.7]{fernández2024ndependent} and compactness). 
%(see Lemma \ref{even odd nip}). 
If $f(x,y)$ has $IP$, then by Ramsey's theorem and compactness, there is an indiscernible sequence $(b_i)_{i<\omega}\subset \C^{\lvert y \rvert}$  for which the sequence $(f(x,b_i))_{i<\omega}$ is independent on $X$.
By indiscernibility, for each $i<\omega$ there is $\sigma_i \in \aut(\C)$ such that $\sigma_i(b_0)=b_i$. Hence, $f_{b_0}$ is not a tame function, because the sequence $(\sigma_i f_{b_0})_{i<\omega}$ is independent.
\end{proof}

\begin{cor}\label{cor: nip iff tame}
    The flow $(\aut(\C), S_{X/E}(\C))$ is tame if and only if $X/E$ is NIP.
\end{cor}
\begin{proof}
    $(\Rightarrow)$ By assumption, for any $f(x,y)\in\mathcal{F}_{X/E}$ and $b\in\C^{\lvert y\rvert}$, $f_b$ is tame, so $f(x,y)$ has NIP by Proposition \ref{tame iff nip formula}. Hence, $X/E$ has NIP by \cite[Lemma 5.7]{fernández2024ndependent}.
    
    $(\Leftarrow)$ By assumption and \cite[Lemma 5.7]{fernández2024ndependent}, every function $f(x,y)\in\mathcal{F}_{X/E}$ has NIP, so for any $b\in \C^{\lvert y \rvert}$ the function $f_b$ is tame by Proposition \ref{tame iff nip formula}. Thus, since by \cite[Proposition 2.1]{10.1215/00294527-2022-0023} the family of functions $\{ f_b: f\in \mathcal{F}_{X/E},b\in \C^{\lvert y \rvert}  \}$ separates points in $S_{X/E}(\C)$, we conclude that $(\aut(\C), S_{X/E}(\C))$ is tame by Fact \ref{separating points: tame}.
    %$(2)\iff (3)$ holds by Proposition \ref{2.3}, Lemma \ref{IP_n hyperdef. set iif formulas} and Fact \ref{separating points: tame} since these results imply that the family $\{ f_b: f\in \mathcal{F}_{X/E},b\in \C^{\lvert y \rvert}  \}$ separates points and every function contained in it is tame.
\end{proof}

%Recall that $E^{\textrm{st}}_{\emptyset}$ is the finest $\emptyset$-type-definable equivalence relation on $\C$ with stable quotient, and $E^{\textrm{NIP}}_{\emptyset}$ is the finest $\emptyset$-type-definable equivalence relation on $\C$ with NIP quotient (which always exist).
Below we will use the notation introduced in Proposition \ref{remark: pist and piNIP} and the comments following it. Note that, while a $\emptyset$-type-definable equivalence relation on $X$ induces a closed, $\aut(\C)$-invariant equivalence relation on $S_X(\C)$, the converse is not true. Namely, not every closed, $\aut(\C)$-invariant equivalence relation on $S_X(\C)$ is created this way. In the case of $\Tilde{E}^{\textrm{WAP}}_{\emptyset}$ and $\Tilde{E}^{\textrm{NIP}}_{\emptyset}$, by Corollaries \ref{cor: stable iff wap} and \ref{cor: nip iff tame}, we get that $S_X(\C)/\Tilde{E}^{\textrm{st}}_{\emptyset}$ is a WAP flow and $S_X(\C)/\Tilde{E}^{\textrm{NIP}}_{\emptyset}$ is a tame flow. However, the next proposition shows that $\Tilde{E}^{\textrm{st}}_{\emptyset}$ may not be equal to $\fwap$ and $\Tilde{E}^{\textrm{NIP}}_{\emptyset}$  may not be equal to $\fta$.

\begin{prop}\label{fwap strictly finer}
    Assume that $X$ is a $\emptyset$-type-definable subset of $\C^\lambda$, and $\C$ is $(\aleph_0+\lambda)$-saturated and strongly $(\aleph_0+\lambda)$-homogeneous.
    \begin{enumerate}
        \item If $X$ is unstable, then $\fwap\subsetneq \Tilde{E}^{\textrm{st}}_{\emptyset}$.
        \item If $X$ has IP, then $\fta\subsetneq \Tilde{E}^{\textrm{NIP}}_{\emptyset}$.
    \end{enumerate}
\end{prop}
\begin{proof}
    The inclusions follow from the above observations that $S_X(\C)/\Tilde{E}^{\textrm{st}}_{\emptyset}$ is WAP and $S_X(\C)/\Tilde{E}^{\textrm{NIP}}_{\emptyset}$ is tame. It remains to show that $\fwap \neq \Tilde{E}^{\textrm{st}}_{\emptyset} $ and $\fta\neq \Tilde{E}^{\textrm{NIP}}_{\emptyset}$. We will prove the first thing; the proof of the second one is analogous.

Since $X$ is unstable, $E^{\textrm{st}}_\emptyset \ne  \;\,=$, so  $\tilde{E}^{\textrm{st}}_\emptyset$ glues some types in $S_X(\C)$ which are realized in $\C$.

On the other hand, define a closed equivalence relation $E$ on $S_X(\C)$ by $$pEq \iff p_= = q_=,$$ where $p_=$ and $q_=$ denote the restrictions of $p$ and $q$ to the empty language (so we allow only the equality relation). Let $S^=_{X}(\C)$ be the collection of all global types in the empty language of the elements from $X(\C')$, where $\C' \succ \C$ is a monster model of the original theory in which $\C$ is small. Then $S^=_{X}(\C) = \{\tp(a/\C)_=: a \in X\} \cup \{\textrm{the unique non-realized type}\}$ is a closed subset of $S^=_{\C}(\C)$ invariant under $\aut(\C)$. We also see that $S_X(\C)/E \cong S^=_{X}(\C)$ as $\aut(\C)$-flows.
As the theory of $\C$ in the empty language is stable, by Corollary \ref{cor: stable iff wap}, we get that $(\Sym(\C), S^=_{\C}(\C))$ is WAP. Hence, since WAP  is closed under decreasing the acting group and under taking subflows, $(\aut(\C), S^=_{X}(\C))$ is also WAP, and so is $(\aut(\C),S_X(\C)/E)$. Therefore, $\fwap \subseteq E$. Thus, since $E$ does not glue any realized types in $S_X(\C)$, neither does $\fwap$.

By the conclusions of the last two paragraphs, we conclude that $\fwap \ne \tilde{E}^{\textrm{st}}_\emptyset$.
\end{proof}

Although $\fwap$ and $\fta$ are almost always strictly finer than $\Tilde{E}^{\textrm{st}}_\emptyset$ and $\Tilde{E}^{\textrm{NIP}}_\emptyset$, the following question and its analog for the tame case remain open:

\begin{question}
    Are the Ellis groups of the flows $$(\aut(\C), S_X(\C)/\fwap)$$ and $$(\aut(\C), S_X(\C)/\Tilde{E}^{\textrm{st}}_\emptyset)$$ isomorphic?
\end{question}

Proposition \ref{fwap strictly finer} justifies our interest in $\fwap$ and $\fta$, because it suggests that the quotients by these equivalence relations should capture more information about the theory in question than the quotients by $\Tilde{E}^{\textrm{st}}_\emptyset$ and $\Tilde{E}^{\textrm{NIP}}_\emptyset$ while maintaining similar good properties (having in mind that WAP is a dynamical version of stability and tameness a dynamical version of NIP).

\appendix
\section{A proof of Proposition \ref{delta is homeomorphism}}\label{Appendix}

To prove Proposition \ref{delta is homeomorphism}, we need Lemma 3.11 from the first arXiv version of \cite{Krupiski2019RamseyTA}. The proof is repeated in \cite[Fact 5.15]{chernikov2024definable}. 
%We include a proof here for completeness.

\begin{lema}\label{tau closure equivalence}
    For any flow $(G,X)$ and $A\subseteq u\mathcal{M}$, the $\tau$-closure $\cl_\tau(A)$ can be described as the set of all limits contained in $u\mathcal{M}$ of the nets $\eta_i a_i$ such that $\eta_i\in\mathcal{M}$, $a_i\in A$, and $\lim_i \eta_i =u$
\end{lema}

%\begin{proof}
%    Consider $a \in \cl_{\tau}(A)$. Then, by the definition of the $\tau$-topology, there are nets $(g_i)_i \subseteq G$ and $(a_i)_i \subseteq A$ such that $\lim_i g_i = u$ and $\lim_i g_i a_i = a$. Note that $u a_i = a_i$, as $a_i \in A \subseteq u\mathcal{M}$. Put $\eta_i := g_i u \in \mathcal{M}$ for all $i$. By left continuity, we have that $\lim_i \eta_i = \lim_i g_i u = (\lim_i g_i) u = u u = u$. Furthermore, $\lim_i \eta_i a_i = \lim_i g_i u a_i = \lim_i g_i a_i = a$.

%Conversely, consider any $a \in u\mathcal{M}$ for which there are nets $(\eta_i)_i \subseteq \mathcal{M}$ and $(a_i)_i \subseteq A$ such that $\lim_i \eta_i = u$ and $\lim_i \eta_i a_i = a$. Since each $\eta_i$ can be approximated by elements of $G$ and the semigroup operation is left continuous, one can find a subnet $(a_j')_j$ of $(a_i)_i$ and a net $(g_j)_j \subseteq G$ such that $\lim_j g_j = u$ and $\lim_j g_j a_j' = a$, which means that $a \in \cl_{\tau}(A)$.
%\end{proof}

\begin{proof}[\textbf{Proof of Proposition \ref{delta is homeomorphism}}]
We first show that $\delta$ is continuous. Consider an arbitrary subbasic closed set $F_{\bar{\varphi}, \bar{p},\bar{q}, r}$ (see Definition \ref{definition: ipp-topology}) and take $\eta \in \cl_\tau(\delta^{-1}[F_{\bar{\varphi}, \bar{p},\bar{q}, r}])$. Choose some nets $(\sigma_i)$ in $\aut(\C)$ and $(\eta_i)$ in $\delta^{-1}[F_{\bar{\varphi}, \bar{p},\bar{q}, r}]$ such that $\lim \sigma_i = u$ and $\lim \sigma_i\eta_i= \eta$. 

Suppose for a contradiction that $ \eta \notin \delta^{-1}[F_{\bar{\varphi}, \bar{p},\bar{q}, r}]$, i.e. $\neg R_{\bar{\varphi}, r}(\eta(\bar{p}),\bar q)$. Then there is $b \models r(y)$ such that $$\varphi_1(x,b) \in \eta(p_1), \dots, \varphi_m(x,b) \in \eta(p_m)$$ and $$\varphi_{m+1}(x,b) \in q_1, \dots, \varphi_{m+n}(x,b) \in q_{n}.$$ 
%%%Krzys: The sentence below was not a sufficient explanation.
%Since $\lim \sigma_i = u$ and $\lim \sigma_i\eta_i= \eta$, there exists $i_0$ such that for all $i>i_0$ we have $\neg R_{\bar{\varphi}, r}(\eta_i(\bar{p}),\bar{q})$, a contradiction with the choice of $\eta_i$.
Then there exists $i_0$ such that for all $i>i_0$ and $j \leq m$ we have $\varphi_i(x,b) \in \sigma_i \eta_i (p_j)$, i.e., $\varphi_i(x,\sigma_i^{-1}(b)) \in \eta_i (p_j)$. On the other hand, since $ \lim \sigma_i =u$, $ \delta(u)=\text{Id}_{\bar J}$, $q_j\in \bar J$, and $\varphi_{m+j}(x,b) \in q_j$ for $j \leq n$, we get that there exists $i_1$ such that for all $i>i_1$ and $j \leq n$ we have $\varphi_{m+j}(x,b) \in \sigma_i(q_j)$, i.e., $\varphi_{m+j}(x,\sigma_i^{-1}(b)) \in q_j$. Then for every $i>\max(i_0,i_1)$ we have $\neg R_{\bar \varphi, r}(\eta_i(\bar p),\bar q)$, a contradictions with the choice of $\eta_i$.

Next, we show the continuity of $\delta^{-1}$. 
%Consider an arbitrary $\eta \in u\mathcal{M}$ and a basis of open neighborhoods of $\delta(\eta)$ consisting of all open neighborhoods of $\delta(\eta)$ of the form $U_{\bar{\varphi}, \bar{p},\bar{q}, r}$, where $U_{\bar{\varphi}, \bar{p},\bar{q}, r}$ is the complement of $F_{\bar{\varphi}, \bar{p},\bar{q}, r}$. 
Consider an arbitrary $\eta \in u\mathcal{M}$ and all open neighborhoods of $\delta(\eta)$ of the form $U_{\bar{\varphi}, \bar{p},\bar{q}, r}$, where $U_{\bar{\varphi}, \bar{p},\bar{q}, r}$ is the complement of $F_{\bar{\varphi}, \bar{p},\bar{q}, r}$. (This is a basis of open neighborhoods of $\delta(\eta)$, but we will not use it.)
Let $I$ consist of all tuples $(\bar{\varphi}, \bar{p},\bar{q}, r, b)$, where the tuples $(\bar{\varphi}, \bar{p},\bar{q}, r)$ are as in Definition \ref{definition: ipp-topology} and $b$ is any tuple realizing $r$ and such that
$$\varphi_1(x,b) \in \delta(\eta)(p_1),\dots, \varphi_m(x,b) \in \delta(\eta)(p_m),
\varphi_{m+1}(x,b) \in q_1,\dots, \varphi_{m+n}(x,b) \in q_n.$$

Order $I$ by:
$(\bar{\varphi}(x,y), \bar{p},\bar{q}, r(y),b) \leq (\bar{\varphi}'(x,z), \bar{p}',\bar{q}', r'(z),b')$ if $y \subseteq z$, $r' |_y=r$, $b'|_y=b$, $m:=|\bar{p}| \leq |\bar{p}'|=:m'$ and $n:=|\bar{q}| \leq |\bar {q}'|=:n'$, and there exists an injection $$\sigma \colon \{1,\dots,m+n\} \to \{1,\dots,m'+n'\}$$ such that $\sigma[\{1,\dots,m\}] \subseteq \{1,\dots,m'\}$, $\sigma[\{m+1,\dots,m+n\}] \subseteq \{m'+1,\dots,m'+n'\}$, and:
\begin{enumerate}
    \item $\varphi_j = \varphi'_{\sigma(j)}$  for all $j \leq m+n$,
    \item $p_j = p'_{\sigma(j)}$ for all $j \leq m$,
    \item $q_{m+j}=q'_{\sigma(m+j)}$ for all $j \leq n$.
\end{enumerate}
It is easy to check that $(I, \leq)$ is a directed set. For $i =(\bar{\varphi}, \bar{p},\bar{q}, r, b) \in I$ by $U_i$ we mean $U_{\bar{\varphi}, \bar{p},\bar{q}, r}$. Clearly $i \leq i'$ implies $U_{i'} \subseteq U_i$.

It suffices to show that for any net $(\eta_k)_{k \in K}$ in $u\mathcal{M}$ such that $\textrm{ipp}$-$\lim_k \delta(\eta_k) = \delta(\eta)$ we have $\tau$-$\lim_k \eta_k =\eta$. For that it is enough to show that for any subnet $(\rho_j)_{j \in J}$ of $(\eta_k)_{k \in K}$ we have that $\eta$ is in the $\tau$-closure of $(\rho_j)_{j \in J}$.

Since $\textrm{ipp}$-$\lim_j \delta(\rho_j) = \delta(\eta)$, for every $i \in I$ there exists $j_i \in J$ such that $\delta(\rho_{j_i}) \in U_i$. Writing $i=(\bar{\varphi}, \bar{p},\bar{q}, r,b)$, we can find $b_i \models r$ such that 
%$$(*)\;\;\;\;\;\;\; \varphi_1(x,b_i) \in \delta(\rho_{j_i})(p_1), \dots,  \varphi_m(x,b_i) \in \delta(\rho_{j_i})(p_m), 
%\varphi_{m+1}(x,b_i) \in q_1, \dots , \varphi_{m+n}(x,b_i) \in q_n.$$
\begin{equation}\tag{$\star$}
\begin{split}
    \varphi_1(x,b_i) \in \delta(\rho_{j_i})(p_1), \dots,  \varphi_m(x,b_i) \in \delta(\rho_{j_i})(p_m), \\
\varphi_{m+1}(x,b_i) \in q_1, \dots , \varphi_{m+n}(x,b_i) \in q_n.
\end{split}
\end{equation}

%\begin{align*}\tag{*}
%    &\varphi_1(x,b_i) \in \delta(\rho_{j_i})(p_1), \dots,  \varphi_m(x,b_i) \in \delta(\rho_{j_i})(p_m)\\
%    &\varphi_{m+1}(x,b_i) \in q_1, \dots , \varphi_{m+n}(x,b_i) \in q_n.
%\end{align*} 

%On the other hand, by the definition of $I$, we have 
%\begin{equation}\tag{$\star\star$}
%    \begin{split}
%        \varphi_1(x,b) \in \delta(\eta)(p_1),\dots, \varphi_m(x,b) \in \delta(\eta)(p_m),\\
%        \varphi_{m+1}(x,b) \in q_1,\dots, \varphi_{m+n}(x,b) \in q_n.
%    \end{split}
%\end{equation}
%$$(**)\;\;\;\;\;\;\;\; \varphi_1(x,b) \in \delta(\eta)(p_1),\dots, \varphi_m(x,b) \in \delta(\eta)(p_m),
%\varphi_{m+1}(x,b) \in q_1,\dots, %\varphi_{m+n}(x,b) \in q_n$$

For each $i\in I$, choose $\sigma_i \in \aut(\C)$ such that $\sigma(b_i)=b$. % (it exists as $b$ and $b_i$ realizes $r$).
Then, by $(\star)$, we have $\varphi_{m+1}(x,b) \in \sigma_i(q_1), \dots , \varphi_{m+n}(x,b) \in \sigma_i(q_n)$.

Since for varying $i \in I$ the tuples $(q_1,\dots,q_n)$ range over all finite tuples from $\bar{\mathcal{J}}$ and the tuples $(\varphi_{m+1}(x,b),\dots, \varphi_{m+n}(x,b))$ over all possible tuples of formulas belonging, respectively, to $q_1,\dots,q_n$, we see that $\lim_i \sigma_i |_{\bar{\mathcal{J}}} = \textrm{Id}_{\bar{\mathcal{J}}}$. So $\lim_i \sigma_i u = u$ and clearly $\sigma_i u \in \mathcal{M}$. 
%(recall that $\pi_{\sigma_i} \in E(S_X(\C))$ is given by $\pi_{\sigma_i}(p) := \sigma_i(p)$).
%(Here $\circ$ denotes the composition of functions; formally it should be $\pi_{\sigma_i}$ in palce of $\sigma_i$; bellow I may skip $\circ$.)

Also by $(\star)$, 
$\varphi_1(x,b) \in \sigma_i (\rho_{j_i} (p_1)), \dots,  \varphi_m(x,b) \in \sigma_i (\rho_{j_i}(p_m))$. Again, since for varying $i \in I$ the tuples $(p_1,\dots,p_m)$ range over all finite tuples from $\bar{\mathcal{J}}$ and the tuples $(\varphi_1(x,b),\dots, \varphi_m(x,b))$ over all possible tuples of formulas belonging, respectively, to $\eta(p_1),\dots,\eta(p_m)$, we see that $\lim_i (\sigma_i \rho_{j_i}) |_{\bar{\mathcal{J}}} = \eta |_{\bar{\mathcal{J}}}$.
Since $\rho_{j_i} = \rho_{j_i}u=u\rho_{j_i}$ and $\eta=\eta u$, we get $\lim_i \sigma_i u \rho_{j_i} = \eta$.

Hence, by the previous two paragraphs and Lemma \ref{tau closure equivalence}, we conclude that $\eta$ is in the $\tau$-closure of $(\rho_j)_{j \in J}$ as required.
\end{proof}

\section{Products of stable and NIP hyperdefinable sets}\label{appendixB}

We will prove that the properties of stability and NIP for hyperdefinable sets are both preserved under (possibly infinite) Cartesian products. They are also clearly closed under taking type-definable subsets. %We consider two $\emptyset$-type-definable sets $X'$ and $Y'$ and two $\emptyset$-type-definable equivalence relations $E$ and $F$ on $X'$ and $Y'$, respectively. We denote by $X$ the hyperdefinable set $X'/E$ and by $Y$ the hyperdefinable set $Y'/F$.
This yields the following:

\begin{cor} \label{finest eq exist for fixed param}
An arbitrary intersection of $\emptyset$-type-definable equivalence relations $(E_i)_{i<\mu}$ with stable [resp. NIP] quotients on a $\emptyset$-type-definable set $X$ is an equivalence relation with stable [resp. NIP] quotient on $X$. Moreover, a finest $\emptyset$-type-definable equivalence relation on $X$ with stable [resp. NIP] quotient always exists.
\end{cor}

\begin{proof}
The first part follows from the fact that stability and NIP are preserved under (possibly infinite) Cartesian products and taking type-definable subsets, and the observation that the hyperdefinable set $$\bslant{X}{\bigcap_{i<\mu}E_i}$$ can be naturally identified with a type-definable subset of $$\prod_{i<\mu}X/E_i.$$
     
     For the moreover part, consider the $\emptyset$-type-definable equivalence relation on $X$ defined as the intersection of all $\emptyset$-type-definable equivalence relations on $X$ with stable [resp. NIP] quotient, and use the first part.
\end{proof}

%Clearly, taking type-definable subsets preserves both stability and NIP. 
First of all, note that directly from Definitions \ref{definition: stability} and \ref{definition: NIP} and compactness, it follows that in order to get preservation of stability and NIP under possibly infinite Cartesian products, it is enough to show preservation of stability and NIP under products of two hyperdefinable sets.
Thus, for the remainder of the appendix, we consider two hyperdefinable sets $X/E$ and $Y/F$ with $X,Y\subseteq \C^\lambda$, where $\lambda$ is smaller than the degree of saturation of $\C$ (note here that without loss of generality we can assume that $X$ and $Y$ leave on tuples of the same length $\lambda$, as if $X \subseteq \C^\alpha$, then for any $\beta$ we can replace $X$ by $X \times \C^\beta \subseteq \C^{\alpha + \beta}$ and $E$ by the relation which is just $E$ on the first $\alpha$ coordinates).

First, we prove preservation of stability. This was stated in \cite[Remark 1.4]{MR3796277}, and the proof we present was suggested by Anand Pillay.

\begin{prop}\label{stable preserved product}
    Let $X/E$ and $Y/F$ by stable hyperdefinable sets. Then, $X/E\times Y/F$ is a stable hyperdefinable set.
\end{prop}

%%%Krzys: It seems we do not need to use a big $\mathcal{I}$ and extracting indiscernibles. It is enough to use $\mathbb{Q}$ and Ramsey's theorem. Do you agree with the version below?
\begin{proof}
    Suppose the conclusion does not hold. Let $(a_i,b_i,c_i)_{i<\omega}$ be an indiscernible sequence witnessing unstability of $X/E\times Y/F$. That is, for all $i,j<\omega$ $a_i\in X/E$, $b_i\in Y/F$, and $\tp(a_i,b_i,c_j)\neq \tp(a_j,b_j,c_i)$ for all $i\neq j$. By compactness, without loss of generality we may replace $\omega$ by $\mathbb{Q}$.

    \begin{claim}
        $(c_j)_{j\neq 0}$ is indiscernible over $a_0$.
    \end{claim}

    \begin{proof}
        Note that it is enough to show that for any $i_1<\dots< i_n<0<i'<j_1< \dots < j_m$ $$c_{i_1},\dots c_{i_n},c_{j_1},\dots,c_{j_m} \equiv_{a_0}c_{i_1},\dots c_{i_{n-1}},c_{i'},c_{j_1},\dots,c_{j_m}.$$ 
%Note also that the sequence $(a_i,c_i)_{i\in (i_{n-1},j_1)}$ is indiscernible over $$K:=\{c_{i_k}: k=1,\dots,n-1 \}\cup \{c_{j_k}: k=1,\dots,m \}.$$
%        Suppose that the conclusion does not hold. This implies that for some setting as above, $$(c_{i_n},a_0)\not\equiv_K (c_{i'},a_0).$$ However, by the indiscernibility of the original sequence, $(c_{i_n},a_0)\equiv_K (c_{0},a_{i'})$. Therefore, the sequence $(a_i,c_iK)_{i\in (i_{n-1},j_1)}$ contradicts the stability of $X/E$.
Suppose that the conclusion does not hold. This implies that for some setting as above we have 
$$(c_{i_n},a_0)\not\equiv_K (c_{i'},a_0),$$
where $$K:=\{c_{i_k}: k=1,\dots,n-1 \}\cup \{c_{j_k}: k=1,\dots,m \}.$$
On the other hand, by the indicernibility of  $(a_i,b_i,c_i)_{i\in \mathbb{Q}}$, we get that $(a_i,c_i)_{i\in (i_{n-1},j_1)}$ is indiscernible over $K$; in particular,  $(c_{i_n},a_0)\equiv_K (c_{0},a_{i'})$. Hence, 
$$(c_{i'},a_0)\not\equiv_K (c_0, a_{i'}).$$ 
Therefore, the sequence $(a_i,c_iK)_{i\in (i_{n-1},j_1)}$ contradicts the stability of $X/E$.
    \end{proof}

    \begin{claim}
        $\tp(a_0,b_0,c_j)$ is constant for $j>0$, $\tp(a_0,b_0,c_j)$ is constant for $j<0$, and $\tp(a_0,b_0,c_1)\neq \tp(a_0,b_0,c_{-1})$.
    \end{claim}

    \begin{proof}
        The fact that it is constant follows from the indiscernibility of the original sequence $(a_i,b_i,c_i)_{i\in \mathbb{Q}}$. Moreover, we have $$\tp(a_0,b_0,c_{-1})\neq \tp(a_{-1},b_{-1},c_0)=\tp(a_0,b_0,c_{1}).$$ 
    \end{proof}

    From the claims, it follows that for any rational numbers  $k$ and $i_1<\dots< i_n<k<j_1< \dots < j_n$ all distinct from $0$, there exists $b'_{k,\bar i,\bar j}\in Y/F$ (where $\bar i=(i_1,\dots,i_n)$ and $\bar j=(j_1,\dots,j_n)$) such that \begin{align*}
    b'_{k,\bar i, \bar j}c_{i_1}\equiv_{a_0}\cdots\equiv_{a_0} b'_{k,\bar i, \bar j}c_{i_n}\equiv_{a_0}b_0c_{-1}\\
    b'_{k,\bar i,\bar j}c_{j_1}\equiv_{a_0}\cdots\equiv_{a_0} b'_{k, \bar i, \bar j}c_{j_n}\equiv_{a_0}b_0c_{1}.
\end{align*}
Thus, by compactness, there exists a sequence $(b_k')_{k \in \mathbb{Q}\setminus \{0\}}$ in $Y/F$ such that for $i \in \mathbb{Q} \setminus \{0\}$ we have: $b_k'c_i \equiv_{a_0} b_0c_{-1}$ when $i<k$, and $b_k'c_i \equiv_{a_0} b_0c_{1}$ when $i>k$. Hence, by compactness and Ramsey's theorem, there is a sequence $(b''_i,c'_i)_{i<\omega}$ with $b_i'' \in Y/F$, which is indiscernible over $a_0$ and $$\tp(b''_i,c'_j/a_0)\neq tp(b''_j,c'_i/a_0)$$ for all $i<j$, contradicting stability of $Y/F$.
\end{proof}

Now, we turn to the NIP case.
See Section \ref{section: Stable vs WAP} for the definition of the family $\mathcal{F}_{X/E}$ and references concerning it.
%%%Krzys: Do we really use here  \cite[Proposition 5.6]{fernández2024ndependent}? It seems it is just a standard compactness argument using the definition of NIP.
%The next characterization of functions of the family $\mathcal{F}_{X/E}$ with NIP easily follows from \cite[Proposition 5.6]{fernández2024ndependent} and compactness.
The next characterization of the functions of the family $\mathcal{F}_{X/E}$ with NIP follows by compactness and Ramsey's theorem.

\begin{lema}\label{even odd nip}
    For any $f(x,y)\in  \mathcal{F}_{X/E}$ the following are equivalent:
    \begin{enumerate}
        \item $f$ has IP.
        \item There exists an indiscernible sequence $(b_i)_{i<\omega}$, $a\in X$, and real numbers $r<s$ such that 
$$ f(a,b_i)\leq r\iff i \text{ is even}, $$ $$f(a,b_i) \geq s\iff i \text{ is odd.}$$
    \end{enumerate}
\end{lema}

%%%Krzys The proof of $(\Leftarrow)$ did not make sense for any linear order, as you used i+1. I changed it to limit ordinals.
\begin{prop}\label{NIP formulas have limit}
%    For every $f(x,y)\in\mathcal{F}_{X/E}$ and for every (infinite) linear order  $\I$ without maximal element, $f(x,y)$ has NIP if and only if for every indiscernible sequence $(b_i)_{i\in\I}$ and $a\in X$ there is $L\in\Image(f)\subseteq [r_1,r_2]$ (for some $r_1,r_2\in \R$) such that for every $\varepsilon>0$ there is an end segment $\I_0\subset \I$  satisfying $$ \lvert f(a,b_i)-L\rvert \leq \varepsilon$$ for all $i\in\I_0$ 
%%%Krzys: \I instead of \I_0 (but it was for the old version, so now it is irrelevant).
%(i.e., $(f(a,b_i))_{i\in\I}$ converges to $L$).
   For every $f(x,y)\in\mathcal{F}_{X/E}$ and for every limit ordinal $\gamma \leq \kappa$ (where $\kappa$ is the degree of saturation of $\C$), $f(x,y)$ has NIP if and only if for every indiscernible sequence $(b_i)_{i<\gamma}$ and $a\in X$ the sequence $(f(a,b_i))_{i<\gamma}$ is convergent.
%i.e., there is $L\in\Image(f)\subseteq [r_1,r_2]$ (for some $r_1,r_2\in \R$) such that for every $\varepsilon>0$ there is an end segment $\I_0\subseteq \gamma$  satisfying $$ \lvert f(a,b_i)-L\rvert \leq \varepsilon$$ for all $i\in\I_0$.
\end{prop}
\begin{proof}
    
$(\Leftarrow)$  Suppose $f(x,y)$ has IP. By Lemma \ref{even odd nip}, there exist an indiscernible sequence $(b_{i})_{i<\omega}$, an element $a\in X$, and $r< s$ such that \begin{align*}
    f(a,b_i) \leq r\iff i \text{ is even},\\
    f(a,b_i) \geq s\iff i \text{ is odd}.
\end{align*}

%By compactness, we may extend the indiscernible sequence $(b_{i})_{i<\omega}$ to a new indiscernible sequence $(b'_{i})_{i< \gamma}$ such that for any $i< \gamma$: if $f(a,b'_i)<r$ then $f(a,b'_{i+1})>s$, and if $f(a,b'_i)>r$ then $f(a,b'_{i+1})<r$.
By compactness, we may extend the indiscernible sequence $(b_{i})_{i<\omega}$ to a new indiscernible sequence $(b'_{i})_{i< \gamma}$ such that for any $i< \gamma$: $f(a,b'_i)\leq r$ if $i$ is even, and $f(a,b'_{i})\geq s$ if $i$ is odd (where $i$ is even if it is an even natural number or an ordinal of the from $\delta +n$ for some limit ordinal $\delta$ and even natural number $n$; odd ordinals are defined analogously). 
%By assumption, there is some $L$ to which $(f(a, b'_i))_{i< \gamma}$ converges. % as $i$ goes to $\infty$. 
%Let $0<\epsilon<\frac{s-r}{2}$. So there is an end segment $\I_0\subseteq \lambda$ such that for all $i\in \I_0$, $\left|f\left(a, b'_i\right)-L\right| \leq \epsilon$. Then, $$\left|f\left(a, b'_{i+1}\right)-L\right| \geq\left|f\left(a, b'_{i+1}\right)-f\left(a, b'_i\right)\right|-\left|f\left(a, b'_i\right)-L\right| \geq(s-r)-\epsilon \geq \frac{s-r}{2}>\epsilon.$$ Which is a contradiction.
This contradicts the assumption that  the sequence $(f(a,b_i'))_{i<\gamma}$ is convergent.

$(\Rightarrow)$  Suppose the conclusion does not hold for an indiscernible sequence $(b_i)_{i <\gamma}$ and $a\in X$. That is, for every $L$ there is some $\epsilon>0$ such that for cofinally many $i<\gamma$ we have $\left|f\left(a, b_i\right)-L\right|>\epsilon$ (formally this means that for every $\beta < \gamma$ there is $i \in (\beta,\gamma)$ such that $\left|f\left(a, b_i\right)-L\right|>\epsilon$).

%$\I$ does not have a maximal element (otherwise, it would converge to that). So ithout loss of generality, let $\I=\omega$. 
Since $\left\{f\left(a, b_i\right) \mid i< \gamma \right\} \subset[r_1,r_2]$ for some real numbers $r_1<r_2$, the sequence $(f\left(a, b_i\right))_{i< \gamma}$ must have some accumulation point $L_{0}$. That is, for any $\epsilon >0$, for cofinally many $i<\gamma$ we have $\left|f\left(a, b_i\right)-L_{0}\right| \leq \epsilon$.

Since $\left(f\left(a, b_i\right)\right)_{i< \gamma}$ does not converge, there is $\epsilon>0$ such that for cofinally many $j<\gamma$ we have $\left|f\left(a, b_j\right)-L_{0}\right|>\epsilon$, and since \(L_{0}\) is an accumulation point, there are cofinally many $i< \gamma$ for which \(\mid f\left(a, b_i\right)-\) $L_{0} \left\lvert\, \leq \frac{\epsilon}{2}\right.$.

%There are infinitely many $j<\omega$ such that $\left|f\left(a_{j}, b\right)-L_{0}\right|>\epsilon$. In particular, 
Thus, there must be either cofinally many $j<\gamma$ such that $f\left(a, b_j\right)>L_{0}+\epsilon$ or cofinally many $j<\gamma$ such that $f\left(a, b_j\right)<L_{0}-\epsilon$. Consider the former case (the latter is analogous). Let $r=L_{0}+\frac{\epsilon}{2}$ and $s=L_{0}+\epsilon$. %In the former case, let $r=L_{0}+\frac{\epsilon}{2}$ and $s=L_{0}+\epsilon$. In the latter, let $r=L_{0}-\epsilon$ and $s=L_{0}-\frac{\epsilon}{2}$.

%Then we will define a subsequence \(\left(c_{i} \mid i<\omega\right)\) of \(\left(a_{i} \mid i<\omega\right)\) which will be indiscernible. 
We now construct an indiscernible sequence $(c_i)_{i<\omega}$ which, together with $a$, will witness that $f(x,y)$ has IP.
Let $c_{0}=b_{i}$ for some $b_{i}$ such that $f(a,b_i) \leq r$. This is possible since there are cofinally many $i<\gamma$ with $f(a,b_{i})$ within \(\frac{\epsilon}{2}\) of \(L_{0}\). Let \(c_{1}=b_{j}\) with $j>i$ be such that $f(a,b_{j}) \geq s$. Similarly, this is possible since there are cofinally many \(j< \gamma\) with $ f\left(a, b_j\right)-L_{0}>\epsilon$. Iterating this process infinitely many times, we get a subsequence $(c_i)_{i<\omega}$ of $(b_i)_{i<\gamma}$ which is indiscernible, %Continue this process, choosing \(c_{n}\) to be \(a_{i_{n}}\) such that if \(c_{n-1}=a_{i_{n-1}}\), then \(i_{n}>i_{n-1}\). This gives us an indiscernible sequence \(\left(c_{i} \mid i<\omega\right) . 
$f\left(a,c_{i}\right) \leq r$ if and only if \(i\) is even, and $f\left(a,c_{i}\right) \geq s$ if and only if \(i\) is odd. Thus, this sequence is as required (by Lemma \ref{even odd nip}).
\end{proof}

%Using the previous results for functions of the family $\mathcal{F}_{X/E}$, we prove the following:

%%%Krzys: Lots of changes and explanations below. All the results to which we refer in our paper and Adrian's paper are stated for finite tuples, not infinite, so I was precise about it below. 
 
\begin{prop}\label{type stabilizes NIP}
    Let $X/E$ be a hyperdefinable set with $X\subseteq \C^\lambda$. If $X/E$ has NIP, for any indiscernible sequence $(b_i)_{i<(\lvert T\rvert +\lambda)^+}$ of tuples from $\C$ of length at most $\lambda$ and any $a/E\in X/E$ there exists $\alpha < (\lvert T\rvert +\lambda)^+$ such that $(b_i)_{\alpha<i <(\lvert T\rvert +\lambda)^+}$ is indiscernible over $a/E$.
\end{prop}

\begin{proof}
    Let $z$ be the tuple of variables corresponding to some/each $b_i$. For a subtuple $y$ of $z$, by $b_i^y$ we will denote the subtuple of $b_i$ corresponding to the variables $y$.

Assume the conclusion does not hold. Then for every $\alpha <(\lvert T\rvert +\lambda)^+$ there is a finite subtuple $y_\alpha$ of $z$, natural number $k_\alpha$, and  two tuples of indices $\alpha<i_1^\alpha<\dots<i_{k_\alpha}^\alpha<(\lvert T\rvert +\lambda)^+$ and $\alpha<j_1^\alpha<\dots<j_{k_\alpha}^\alpha<(\lvert T\rvert +\lambda)^+$ such that 
$$\tp(a/E,b_{i_1^\alpha}^{y_\alpha},\dots, b_{i_{k_\alpha}^\alpha}^{y_\alpha}) \ne \tp(a/E,b_{j_1^\alpha}^{y_\alpha},\dots, b_{j_{k_\alpha}^\alpha}^{y_\alpha}).$$
Thus, by \cite[Proposition 2.1]{10.1215/00294527-2022-0023}, there exists a function 
%$f_\alpha(x,z_1,\dots,z_{k(\alpha)})\in \mathcal{F}_{X/E}$, where $|z_i| = |y_\alpha|$, 
$f_\alpha \colon \C^\lambda \times \C^{k_\alpha|y_\alpha|} \to \mathbb{R}$ from the family $\mathcal{F}_{X/E}$ and rationals $r_\alpha<s_\alpha$ satisfying 
$$f_\alpha(a,b_{i_1^\alpha}^{y_\alpha},\dots,b_{i_{k_\alpha}^\alpha}^{y_\alpha})\leq r_\alpha \wedge f_\alpha(a,b_{j_1^\alpha}^{y_\alpha},\dots,b_{j_{k_\alpha}^\alpha}^{y_\alpha}) \geq s_\alpha.$$
On the other hand, as explained in the proof of \cite[Corollary 2.4]{10.1215/00294527-2022-0023}, for every $m \in \omega$ the space of all functions $\C^\lambda \times \C^m \to \mathbb{R}$ from the family $\mathcal{F}_{X/E}$  with respect to the supremum norm has a dense subset $\mathcal{F}_m$ of cardinality $\lvert T\rvert +\lambda$. It is then clear that for every $\alpha <(\lvert T\rvert +\lambda)^+$ the function  $f_\alpha$ above can be chosen from the family $\mathcal{F}_{k_\alpha |y_\alpha|}$.

Therefore, there is a finite subtuple $y$ of $z$, natural number $k$, function $f \colon \C^\lambda \times \C^{k|y|} \to \mathbb{R}$ from the family $\mathcal{F}_{X/E}$, and rationals $r<s$ such that $k_\alpha=k$, $y_\alpha = y$, $f_\alpha=f$, $r_\alpha=r$, and $s_\alpha=s$ for cofinally many $\alpha <(\lvert T\rvert +\lambda)^+$. Then, we can construct inductively a sequence $\II=(i_1^l,\dots,i_k^l)_{l<\omega}$ such that $i_1^l<\dots<i_k^l<i_1^{l+1}$ for all $l<\omega$ and: \begin{itemize}
        \item $f(a,b_{i^l_1}^y,\dots,b_{i^l_{k}}^y)\leq r \iff l \text{ is even}$,
        \item $f(a,b_{i^l_1}^y,\dots,b_{i^l_{k}}^y)\geq s \iff l \text{ is odd}$.
    \end{itemize}
    As the sequence $(b_{i_1^l}^y,\dots, b_{i_k^l}^y)_{l<\omega}$ is indiscernible, this implies (by Lemma \ref{even odd nip}) that the function $f$ has IP. 
%%%Krzys: Make sure that Lemma 5.7 from Adrian's paper works with $y_0$ being any finite tuple (not just a single variable). I guess this is the context there, but it is not written anywhere.
By \cite[Lemma 5.7]{fernández2024ndependent}, this is a contradiction with the assumption that $X/E$ has NIP.
\end{proof}

\begin{cor}\label{nip preserved product}
    Let $X/E$ and $Y/F$ be hyperdefinable sets with $NIP$. Then $X/E\times Y/F$ has NIP.
\end{cor}

%%%Krzys: I also added more details below.
\begin{proof}
Recall that without loss of generality, $X,Y \subseteq \C^\lambda$.
    Let $(c_i)_{i<(\lvert T\rvert +\lambda)^+}$ be an arbitrary indiscernible sequence of finite tuples from $\C$, and $(a/E,b/F)$ an arbitrary pair from $X/E\times Y/F$. 

By Proposition \ref{type stabilizes NIP}, there is $\alpha < (\lvert T\rvert +\lambda)^+$ such that $(c_i)_{\alpha <i<(\lvert T\rvert +\lambda)^+}$ is indiscernible over $a/E$. Hence, by compactness and Ramsey's theorem, we can find $a'Ea$ for which $(c_ia')_{\alpha <i<(\lvert T\rvert +\lambda)^+}$ is indiscernible. Applying Proposition \ref{type stabilizes NIP} again, but this time to the sequence $(c_ia')_{\alpha <i<(\lvert T\rvert +\lambda)^+}$, we get that there exists $\beta< (\lvert T\rvert +\lambda)^+$ such that $(c_ia')_{\beta <i<(\lvert T\rvert +\lambda)^+}$ is indiscernible over $b/F$. Then $(c_i)_{\beta <i<(\lvert T\rvert +\lambda)^+}$ is indiscernible over $(a/E,b/F)$. Hence, for every $f(x,y,z)\in \mathcal{F}_{X/E\times Y/F}$ with $|z|=|c_i|$, the sequence $(f(a,b,c_i))_{\beta <i<(\lvert T\rvert +\lambda)^+}$ is constant. Therefore, since $a/E$, $b/F$, and the sequence $(c_i)_{i<(\lvert T\rvert +\lambda)^+}$ were arbitrary, using Proposition \ref{NIP formulas have limit}, we conclude that every $f(x,y,z)\in \mathcal{F}_{X/E\times Y/F}$ has NIP. Therefore, $X/E\times Y/F$ has NIP by \cite[Lemma 5.7]{fernández2024ndependent}.

%By applying Proposition \ref{type stabilizes NIP} twice, we get that there is $\alpha < (\lvert T\rvert +\lambda)^+$ such that $(c_i)_{\alpha <i<(\lvert T\rvert +\lambda)^+}$ is indiscernible over $(a/E,b/F)$.  Thus, the type of $(a/E,b/F,c_i)_{\alpha <i<(\lvert T\rvert +\lambda)^+}$ is constant. Since $a/E$, $b/F$ and the sequence $(c_i)_{i<(\lvert T\rvert +\lambda)^+}$ were arbitrary, Proposition \ref{NIP formulas have limit} implies that every $f(x,y,z)\in \mathcal{F}_{X/E\times Y/F}$ has NIP. Therefore, $X/E\times Y/F$ has NIP by \cite[Lemma 5.7]{fernández2024ndependent}.
\end{proof}
%By the conclusions of the last two paragraphs and Lemma 3.11 from the first arXiv version of  https://arxiv.org/pdf/1912.10527v1, 
\printbibliography
%\nocite{*}
\end{document}